\documentclass[11pt]{amsart}

\usepackage{amssymb}
\usepackage{amsthm}
\usepackage{amsmath}
\usepackage{graphicx}
\usepackage{comment}
\usepackage[margin = 1in]{geometry}
\usepackage{enumerate}

\usepackage{lineno}
\linenumbers

 \newtheorem{thm}{Theorem}[section]
 
 \newtheorem{lem}[thm]{Lemma}
 \newtheorem{prop}[thm]{Proposition}

 \theoremstyle{definition}

 \theoremstyle{remark}
 \newtheorem{rem}{Remark}

 \numberwithin{equation}{section}

\DeclareMathOperator{\PV}{PV}

\newcommand{\per}{\text{per}}
\newcommand{\loc}{\text{loc}}
\newcommand{\ud}{\,\mathrm{d}}

\newcommand{\Erho}{E_{\rho}}
\newcommand{\Eu}{E_{u}}
\newcommand{\Ephi}{E_{\phi}}
\newcommand{\Eh}{E_{h}}

\newcommand{\Eud}{E_{u}^{\delta}}

\newcommand{\Phib}{\Phi_{b}}

\newcommand{\Phibp}{{\Phi_{b}}'}

\newcommand{\wps}{\per^\star}

\newcommand{\dpstyle}{\displaystyle}

\newcommand{\lpi}{\frac{2\pi}{L}}
\newcommand{\ale}{a.e.\,\,}

\begin{document}

\title[Continuum limit of a mesoscopic model of step
  motion on vicinal surfaces]{Continuum limit of a mesoscopic model with elasticity of step
  motion on vicinal surfaces}

\author{Yuan Gao}
\address{School of
Mathematical Sciences\\
   Fudan  University,
Shanghai 200433, P. R.\ China\\Department of Mathematics and Department of
  Physics\\Duke University,
  Durham NC 27708, USA}
\email{gaoyuan12@fudan.edu.cn}
\author{Jian-Guo Liu}
\address{Department of Mathematics and Department of
  Physics\\Duke University,
  Durham NC 27708, USA}
\email{jliu@phy.duke.edu}

\author{Jianfeng Lu}
\address{Department of Mathematics, Department of
  Physics, and Department of Chemistry\\Duke University, Box 90320,
  Durham NC 27708, USA}
\email{jianfeng@math.duke.edu}

\date{\today}

\begin{abstract}
  This work considers the rigorous derivation of continuum models of
  step motion starting from a mesoscopic
Burton-Cabrera-Frank (BCF) type model following the
  work [Xiang, SIAM J. Appl. Math. 2002]. We prove that as the lattice
  parameter goes to zero, for a finite time interval, a modified
  discrete model converges to the strong solution of the limiting PDE
  with first order convergence rate.
\end{abstract}

\maketitle

\section{Introduction}

In this work, we revisit the derivation of continuum model for step
flow with elasticity on vicinal surfaces. The starting point is the
Burton-Cabrera-Frank (BCF) type models for step flow \cite{BCF}; see
\cite{Duport1995a, Duport1995b, Tersoff1995,Tersoff1998} for extensions
to include elastic effects. These are mesoscopic models which
track the position of each individual step (and hence keep the
discrete nature of the step fronts), while adopt a continuum
approximation for the interactions of the steps with surrounding atoms
of the thin film. The step motion is hence characterized by a system
of ODEs. Such models are widely used for crystal growth of thin films
on substrates, with many scientific and engineering applications
\cite{PimpinelliVillain:98, WeeksGilmer:79, Zangwill:88}.  The goal of
this work is to rigorously understand the PDE limit of such models.

To avoid unnecessary technical difficulties, we will study a periodic
train of steps in this work. Denote the step locations at time $t$ by
$x_i(t), i \in \mathbb{Z}$, we assume that
\begin{equation}\label{1}
  x_{i+N}(t) - x_{i}(t) = L, \qquad \forall\,i\in \mathbb{Z},\,\forall\,t\geq 0,
\end{equation}
where $L$ is a fixed length of the period. Thus, only the step
locations in one period $\{x_i(t), \, i = 1, \ldots, N\}$ are
considered as degrees of freedom, see Figure 1 for example.

\begin{figure}[htbp]
\includegraphics[width=6in]{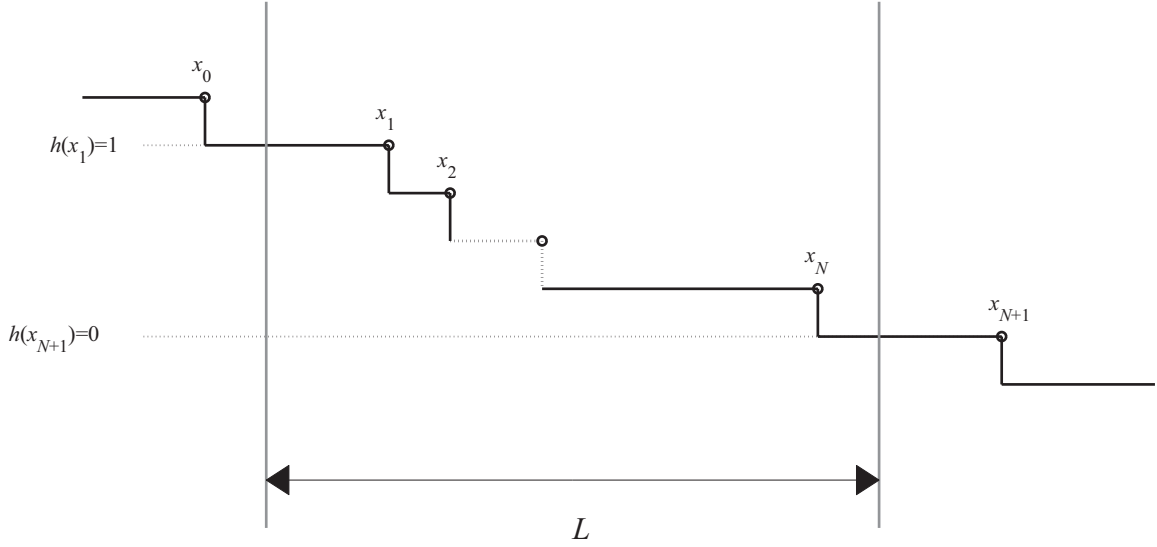}
\caption{ An example of one periodic steps.}
\end{figure}

We denote the height of each step as $a = \frac{1}{N}$, and thus the total
height change across the $N$ steps in the period is given by $1$. Corresponding to the step locations, we define the height profile $h_N$ of the steps as
\begin{equation}\label{20_1.2}
  h_N(x, t) = \frac{N-i}{N}, \quad \text{for } x \in [x_i(t), x_{i+1}(t)), \,\, i = 1, \ldots, N.
\end{equation}
Moreover, $h_N$ can be further extended, consistent with the periodic assumption \eqref{1}, such that
\begin{equation}\label{2}
  h_N(x+L)-h_N(x)= -1,\qquad \forall\,x\in \mathbb{R}.
\end{equation}
For the continuum limit, we consider the step height $a \to 0$ or
equivalently, the number of steps in one period $N \to \infty$.

In the pioneering work \cite{Xiang2002} (see also \cite{Xiang2004}),
\textsc{Xiang} considered a BCF type model which incorporates the
elastic interaction as\footnote{Compared to \cite{Xiang2002}, we drop
  all the physical constants that are mathematically unimportant.}
\begin{equation}\label{3}
  \frac{\ud x_i}{\ud t}=a^2 \Bigl(\frac{f_{i+1}-f_{i}}{x_{i+1}-x_i}-\frac{f_{i}-f_{i-1}}{x_{i}-x_{i-1}}\Bigr),\quad i=1,\cdots,N,
\end{equation}
where $f_i$'s are the local chemical potential given by
\begin{equation*}
  f_i := \frac{\partial E}{\partial x_i} =-\sum_{j\neq i}\Bigl(\frac{\alpha _1}{x_j-x_i}-\frac{\alpha _2}{(x_j-x_i)^3}\Bigr),
\end{equation*}
with the parameters $\alpha_1 = \frac{4}{\pi} a^4$, { $\alpha_2 = \frac{2}{\pi} a^6$} and the energy functional $E$ given by
\begin{equation*}
  E = \frac{1}{2}\sum_{i=1}^N\sum_{j\neq i} \Bigl(\alpha_1\ln \lvert x_i-x_j\rvert + \frac{\alpha_2}{2} \frac{1}{(x_i-x_j)^2}\Bigr).
\end{equation*}
For the limit $a\rightarrow 0$, \textsc{Xiang}~\cite{Xiang2002}
asymptotically derived the corresponding continuum model
\begin{equation}\label{4}
  h_t= \pi \alpha_1 a^2 \Bigl(- H(h_x)+ \frac{1}{2\pi}  \frac{ah_{xx}}{h_x}
  + \frac{\pi}{2} \frac{\alpha_2}{\alpha_1} \frac{h_x h_{xx}}{a} \Bigr)_{xx}.
\end{equation}
Here $H(\cdot)$ is the $L$-periodic Hilbert transform:
\begin{equation}
  (H u)(x) :=\frac{1}{L}\PV\int_{0}^{L}u(x-s)\cot(\frac{\pi s}{L}) \ud s.
\end{equation}
Observe that for the particular choice of the parameters $\alpha_1$
and $\alpha_2$, \eqref{4} suggests to rescale $t$ to consider time
scale of the order $O(a^{-6})$.  Moreover, {  the coefficients in
  front of the term $h_x h_{xx}$ and the term $\frac{h_{xx}}{h_x}$ in
  the bracket scale as $a$ so they become higher order terms compared
  with the first one. As argued in \cite{Xiang2004}, the term
  $a\frac{h_{xx}}{h_x}$ is the correction to the misfit elastic energy
  density due to the discrete nature of the stepped surface. Although
  it is small compared to the leading-order term $H(h_x)$, it is
  comparable with the term $ah_xh_{xx}$, which comes from the broken
  bond elastic interaction between steps.  When formally ignoring
  these terms with small $a$-dependent amplitude, the PDE analysis for
  $h_t=-H(h_x)_{xx}$ is easy because the operator $H(\cdot)_x$ is a
  negative operator.  }

Recently, motivated by the PDE \eqref{4} proposed by \cite{Xiang2002},
\textsc{Dal Maso}, \textsc{Fonseca} and \textsc{Leoni}
\cite{Leoni2014} studied the weak solution of \footnote{For the
  convenience of calculation, we set the coefficients slightly
  different from \cite{Leoni2014}. Moreover, instead of taking $h$ to
  be increasing as in \cite{Leoni2014}, we take $h$ to be decreasing
  corresponding to physical interpretation of $h$ being the height of
  the vicinal surface, which is the same convention as
  \cite{Xiang2002,Xiang2004}.}
\begin{equation}\label{5}
  h_t =  \Bigl(-\lpi H(h_x)+\big(3h_x+\frac{1}{h_x}\big)h_{xx}\Bigr)_{xx},
\end{equation}
in terms of a variational inequality. Note that all the coefficients
in this PDE are $O(1)$, unlike the PDE \eqref{4}. They validated \eqref{5} analytically
by verifying the positivity of $h_x$. Rather remarkably,
they found an approximation problem and proved the limit of
the solution to the approximation problem also satisfies the weak
version of variational inequality, which is satisfied by strong
solution. Moreover, \textsc{Fonseca}, \textsc{Leoni} and \textsc{Lu}
\cite{Leoni2015} obtained the existence and uniqueness of the weak
solution. They applied Rothe method
and truncation method to carefully deal with the singularity term.

\smallskip

Our goal is to rigorously prove the continuum limit of BCF type
models for step flow. While it would be nice to recover \eqref{4} using
the scaling considered in \cite{Xiang2002}, it is quite challenging
(if not impossible) since the PDE \eqref{4} involves {  two scales},
correspond to the three terms on the right hand side:
\begin{align*}
  O(1): \quad H(h_x); \qquad {  O(a)}: \quad h_x h_{xx}; \qquad
  O(a): \quad \frac{h_{xx}}{h_x}.
\end{align*}
Instead, we follow the scaling of the PDE \eqref{5} considered in
\cite{Leoni2014, Leoni2015}. We will derive \eqref{5} as the continuum
limit from a slightly modified BCF type mesoscopic model: we consider
the step-flow ODE \eqref{3} with a rescaled time, i.e.,
\begin{equation}\label{ODE22}
\frac{\ud x_i}{\ud t}=\frac{1}{a} \biggl(\frac{f_{i+1}-f_{i}}{x_{i+1}-x_i}-\frac{f_{i}-f_{i-1}}{x_{i}-x_{i-1}}\biggr),\quad i=1,\cdots,N,
\end{equation}
 with a modified
chemical potential
\begin{equation}\label{622_1}
  f_i:= -\frac{2}{L}\sum_{j\neq i}\frac{a}{x_j-x_i}+\biggl(\frac{1}{x_{i+1}-x_i}-\frac{1}{x_{i}-x_{i-1}}\biggr) +\biggl(\frac{a^2}{(x_{i+1}-x_i)^3}-\frac{a^2}{(x_i-x_{i-1})^3}\biggr);
\end{equation}
see Section \ref{sec5}.  The first term in $f_i$ comes from the misfit
elastic interaction between the steps, which is an attractive
interaction. The second and third terms come from the broken bond
elastic interaction between steps, which are repulsive
terms. Different from \textsc{Xiang}'s chemical potential in
\cite{Xiang2002}, we choose the scaling so that the attractive and
repulsive interactions have the same order as $a \to 0$. We add the
repulsive term $\frac{1}{x_{i+1}-x_i}-\frac{1}{x_{i}-x_{i-1}}$ to
cancel a singularity from the first term, which seems to be necessary.
Moreover, to ease the mathematical derivation, we restrict the
repulsive terms to the nearest neighbor, which is the dominant
contribution.

Our modified ODE system, from both the view of chemical potential and
free energy, is balanced in order. Therefore unlike the original ODE
systems which (at least heuristically) lead to a PDE with multiple
scales, our system converges to PDE \eqref{5} in the limit. We
are also able to obtain the convergence rate of order $a$ for
local strong solution of the continuum PDE.

For the study of the PDE \eqref{5}, we discover four variational
structures with four corresponding energy functionals, in terms of
step height $h$, step location $\phi$, step density $\rho$ and
anti-derivative of $h$, denoted as $u$.
{  Those four kinds of descriptions are equivalent rigorously for strong local solution but it is convenient to use different one when studying different aspects of our problem. The height $h$ is the original variable indicating the evolution of surface height while it is a better idea to use $\rho$ and $u$ to study the strong local solution of continuum model \eqref{5} due to its concise variational structure.}
  In the proof of convergence
rate in Section \ref{sec5}, \ref{consis} and \ref{sec7}, {  since the original discrete model is described by each step location $x_i$,} it is more
natural to use the variational structure of step location $\phi$,
which is the inverse function of step height $h$, i.e.
\begin{equation}\label{phi}
  \alpha = h(\phi(\alpha, t), t), \quad \forall\, \alpha.
\end{equation}


For the properties of local strong solution of continuum PDE
\eqref{5}, we used the variational structures for $u$ and $\rho$ to
establish some \textit{a-priori} estimates and then obtain the
existence and uniqueness for local strong solution to the continuum
PDE; see Section \ref{sec4}. We state the main result of Section
\ref{sec4} below, with the notations $I:=[0,L]$,
\begin{equation}\label{626_1}
W^{k,p}_{\wps}(I):=\{u(x)\in W_{\loc}^{k,p}(\mathbb{R});\, u(x+L)-u(x)=-1\},
\end{equation}
and
\begin{equation}\label{626_2}
W^{k,p}_{\per_0}(I):=\{u\in W^{k,p}(I); u \text{ is $L$-periodic and mean value zero in one period}\}.
\end{equation}
Standard notations for Sobolev spaces are assumed above.

\begin{thm}\label{local_h}
Assume $h^0\in W^{m,2}_{\per^\star}(I),$ $h_{x}^0\leq \beta$, for some constant $\beta<0$, $m\in \mathbb{Z},\, m\geq 6$. Then there exists time $T_{m}>0$ depending on $\beta,\,\|h^0\|_{W^{m,2}_{\per^\star}}$
such that
$$h\in L^{\infty}([0,T_m];W^{m,2}_{\per^\star}(I))\cap L^2([0,T_m];W^{m+2,2}_{\per^\star}(I))\cap C([0,T_m];W^{m-4,2}_{\per^\star}(I)),$$
$$h_t\in L^{\infty}([0,T_m];W^{m-4,2}_{\per}(I))$$
is the unique strong solution of \eqref{5} with initial data $h^0$, and $h$ satisfies
\begin{equation}\label{620_3}
h_x\leq\frac{\beta}{2},\quad \ale  t\in[0,T_m],\,x\in [0,L].
\end{equation}
\end{thm}
Moreover, we also study the stability of the
linearized $\phi$-PDE. This is important in the construction of
approximate solutions to the PDE with high-order consistency, which is
crucial in the proof of convergence.

For the convergence result of mesoscopic model, we first testify our modified ODE system has a global-in-time solution; see Proposition \ref{prop625}. More explicitly, we prove that the steps and terraces will keep monotone if we have monotone initial data. This is consistent with the positivity of step density $\rho$ of the PDE. Then we calculate the consistency of the step location continuum equation and ODE system till order $a$; see Theorem \ref{th4.4}. However, due to the nonlinearity and fourth order derivative in our problem, we need to utilize \textit{a-priori} assumption method and construct an auxiliary solution with high-order consistency. By establishing the stability of the linearized ODE system and carefully calculating the Hessian of coefficient matrix of ODE system, which is a 3rd-order tensor, we finally get the convergence rate $O(a)$ of modified ODE system to its continuum PDE limit.

Recall the definition \eqref{20_1.2} and \eqref{phi}. Denote
$$\alpha_i=h(x_i(0),0)=\frac{N-i}{N},$$
and
$$\phi_i(t)=\phi(\alpha_i,t).$$
We state the main convergence result in this work as follows:
\begin{thm}\label{thm_main}
Let the step height be $a=\frac{1}{N}$. Assume for some constant $\beta<0$, some $m\in\mathbb{N}$ large enough, the initial datum $h(0)\in W^{m,2}_{\per^\star}(I)$ satisfies
\begin{equation}\label{h_initial}
 h_x(0)\leq \beta <0.
\end{equation}
Let $h(x,t)$ be the exact solution of \eqref{5} on $[0,T_m],$ where $T_m$ is the maximal existence time for strong solution defined in Theorem \ref{local_h}. Let $\phi(\alpha,t)$ be the inverse function of $h(x,t)$ defined in \eqref{phi}, whose nodal values are denoted as $\phi_N(t):=\{\phi(\alpha_i,t),\,i=1,\cdots,N\}.$
Let $x(t)=(x_1(t),\cdots,x_N(t))$ be the solution to ODE \eqref{ODE22} with $f_i$ defined in \eqref{622_1} and initial data $x(0)=\phi_N(0)$.
Then there exists $N_0$ large enough such that for $N>N_0$, we have $x(t)$ converges to $\phi(\alpha,t)$ with convergence rate $a$, in the sense of
\begin{equation}\label{620main}
\|x(t)-\phi_N(t)\|_{\ell^2}\leq C(\beta,\|h^0\|_{W^{m,2}_{\per^\star}})a, \text{ for }t\in[0,T_m],
\end{equation}
where $C(\beta,\|h^0\|_{W^{m,2}_{\per^\star}})$ is a constant depending only on $\beta$ and $\|h^0\|_{W^{m,2}_{\per^\star}}$.
\end{thm}

Several remarks of the main result are in order.

\begin{rem}
In fact, we can achieve a better convergence rate $O(a^2)$, if $f_i$ is modified to be
\begin{equation*}
\begin{aligned}
\widetilde{f}_i:=&-\frac{2}{L}\sum_{j\neq i}\frac{a}{x_j-x_i}+\Big(1-\frac{a}{2}\Big)\biggl(\frac{1}{x_{i+1}-x_i}-\frac{1}{x_{i}-x_{i-1}}\biggr)
+\biggl(\frac{a^2}{(x_{i+1}-x_i)^3}-\frac{a^2}{(x_i-x_{i-1})^3}\biggr).
\end{aligned}
\end{equation*}
Compared with \eqref{622_1}, the coefficient of the second term is
changed from $1$ to $1-\frac{a}{2}$. This is done to better correct
the error from the discretization of the Hilbert transform as $a \to
0$ (recall the second term in \eqref{622_1} is introduced to correct
the singularity from the first term). In fact, by Lemma \ref{l4.1}, we
know the leading error
$\frac{a}{2}\frac{\phi_{\alpha\alpha}}{\phi_\alpha^2}$ in Lemma
\ref{l4.2} can be removed by such a correction term. Hence we can get
$O(a^2)$ consistency in Section \ref{consis}, and consequently, the
convergence rate can be improved to $O(a^2)$ in Theorem \ref{thm_main}
for the modified microscopic model.
\end{rem}

\begin{rem}
  Theorem~\ref{thm_main} is a result of local convergence to strong
  solutions to the PDE. The global convergence of the ODE system to
  the (weak) global-in-time solution to the PDE \eqref{5} is more
  challenging and will be left for the future. We hope the additional
  understanding of the variational structures of the PDE \eqref{5}
  provided in this work would help the future investigation on global
  convergence.
\end{rem}

\begin{rem}
  To avoid unnecessary technical complications and to make the
  presentation of the convergence result clear, in this work we do not
  try to optimize the initial regularity that is needed in the
  Theorem~\ref{thm_main}. We just set $m$ to be large enough, so that
  we may assume sufficient regularity of the solution.
\end{rem}

While a comprehensive review of the vast literature of crystal growth
is beyond the scope of this work, let us review here some related
works mostly in the mathematical literature.  Besides the work of
\cite{Xiang2002}, the derivation of the continuum limit of BCF models
have also been considered in other works, see e.g.,
\cite{Tang1997,Yip2001,Shenoy2002,Margetis2011}. However, as far as we
know, the derivation has not been done on the rigorous level and
moreover, the convergence rate is provided here, which seems to be
missing before in the literature. {  The idea using step location for
formal asymptotic analysis was inspired by \cite{Xiang2002}. In order to
get the convergence rate rigorously, we find it is better to first study
the continuum PDE for the inverse function $\phi$, instead of the height $h$.} Recently, in the
attachment-detachment-limited (ADL) regime, \textsc{Al Hajj Shehadeh, Kohn and
  Weare} \cite{She2011} studied the continuum limit of self-similar
solution and obtained the convergence rate. Related to the stability
analysis, the linear stability of thin film (known as the ATG
instability) has been analyzed in previous works, see e.g.,
\cite{Xiang2004,Grin1986,Sro1989}. While we consider here the one
spatial dimensional models, the asymptotic derivation of two
dimensional continuum models have been considered in \textsc{Margetis
  and Kohn} \cite{Margetis2006} and \textsc{Xu and Xiang}
\cite{Xiang2009}, the rigorous aspects of these results will be
interesting future research directions.

For the discrete BCF model considered in \cite{Xiang2002}, very
recently, \textsc{Luo, Xiang and Yip} \cite{Luo2016} rigorously proved
the step bunch phenomenon, which characterized the limiting behavior of
the system as $t \to \infty$. They have also connected the step
bunching with continuum models through a $\Gamma$-convergence argument
\cite{Luoinprep}.  These works motivate further study of the
continuum limit of mesoscopic models of crystal
growth.

Let us also mention that while our starting point is step flow models,
the derivation of the continuum limit can be also considered starting
from a more atomistic description, such as a kinetic Monte Carlo type
model. See the works \cite{Guo:88, Yau:91, FunakiSpohn:97,
  Nishikawa:02} and more recently \cite{MarzuolaWeare2013}. See also a
recent work that aims to derive BCF type models from a kinetic Monte
Carlo lattice model \cite{LuLiuMargetis:15}.

The rest of this paper is organized as follows. In Section \ref{sec2},
after setting up some notations, we introduce four equivalent forms of
continuum PDE \eqref{5} and their variational structures. Section
\ref{sec4} is devoted to establish the existence, uniqueness and
stability for local strong solution of the PDE. We then introduce the
modified step-flow ODE in Section \ref{sec5}, and state the global
existence result for the modified ODE system.  Section \ref{consis} is
devoted to prove the consistency result for ODE system and its
continuum limit PDE. Finally, by constructing an auxiliary solution
with high-order consistency, we obtain the convergence rate of the
modified ODE to its continuum PDE limit in Section \ref{sec7}, which
completes the proof of our main result Theorem \ref{thm_main}.

\section{The continuum model}\label{sec2}

In this section, we discuss the properties of the continuum model.
Besides using the height profile $h$, it would be useful to rewrite
the dynamics in a few equivalent ways. Let us introduce the following definitions
\begin{itemize}
\item step location $\phi(\alpha, t)$, the inverse function of $h$:
  \begin{equation*}
    \alpha = h(\phi(\alpha, t), t), \quad \forall\, \alpha;
  \end{equation*}
\item step density $\rho(x, t)$, the (negative) gradient of $h$:
  \begin{equation}
    \rho(x, t) = - h_x(x, t);
  \end{equation}
\end{itemize}

\begin{itemize}
\item $u(x,t)$, the (negative) anti-derivative of $h$:
  \begin{equation}
    h(x, t) = - u_x(x, t)-bx-k_0,
  \end{equation}
\end{itemize}
where $b,\,k_0$ are constants chosen to guarantee the periodicity of $u_x$.

Now we establish the variational structures for $h,\,u,\,\rho,\,\phi$. In Section \ref{sec4}, it will be convenient to use $\rho$-equation and $u$-equation, while it will be proper to use $\phi$-equation when studying the continuum limit in Section \ref{sec5}, \ref{consis}, \ref{sec7}.

\subsection{Equation for height profile $h$}

Let us consider the PDE for the height profile

\begin{equation*}
  h_t=  \Bigl(-\lpi H(h_x)+\big(3h_x+\frac{1}{h_x}\big)h_{xx}\Bigr)_{xx}.
\end{equation*}
As mentioned in Introduction, the coefficients here are
independent of $a$. In
Section~\ref{consis}, we will show that this continuum PDE can be
derived as the limit of a BCF type discrete atomistic model.

%

First we observe that the evolution equation \eqref{5} has a
variational structure.  Define the total energy $\Eh$ as a functional
of $h$:
\begin{equation}
\Eh(h):=\int_{0}^{L}\biggl(\frac{1}{L}\int_{0}^{L}\ln \bigl\lvert \sin(\frac{\pi}{L}(x-y)) \bigr\rvert h_x h_y \ud y -h_x\ln(-h_x)-\frac{h_x^3}{2}\biggr) \ud x.
\end{equation}
Then we have
\begin{equation}\label{h}
h_t=\mu_{xx}=\biggl(\frac{\delta \Eh}{\delta h}\biggr)_{xx},
\end{equation}
where the chemical potential $\mu$ is given by
\begin{equation}
  \mu := \frac{\delta \Eh}{\delta h} = -\PV\int_0^L \frac{2\pi}{L^2}\cot\frac{\pi(x-y)}{L}h_y(y) \ud y+\frac{h_{xx}}{h_x}+3h_x h_{xx}.
\end{equation}
To see this, let us calculate in Lemma~\ref{l3.0} the functional
derivative $\frac{\delta \Eh^0}{\delta h}$ for
\begin{equation}
  \Eh^0(h):=\int_0^L\int_0^L\ln\big\lvert\sin \frac{\pi(x-y)}{L}\big\rvert h_x h_y \ud x \ud y.
\end{equation}
The derivative of the other two terms in $E_h$ is straightforward.

\begin{lem}\label{l3.0}
Assume $h(x)\in C^2([0,L])$.
We have
$$\frac{\delta \Eh^0}{\delta h}=-\PV\int_0^L\frac{2\pi}{L}\cot\frac{\pi(x-y)}{L} h_y(y) \ud y .$$
\end{lem}
\begin{proof}
First denote
$$\Eh^\delta(h):=\int_0^L\bigg(\int_0^{x-\delta}+\int_{x+\delta}^L\bigg)\ln\big\lvert\sin \frac{\pi(x-y)}{L}\big\rvert h_x h_y \ud y \ud x .$$
By the definition of the principal value integral, we have
$$\frac{\ud }{\ud \varepsilon}\biggr|_{\varepsilon=0}\Eh^0(h+\varepsilon\tilde{h})
=\frac{\ud }{\ud
  \varepsilon}\biggr|_{\varepsilon=0}\lim_{\delta\rightarrow
  0^+}\Eh^\delta(h+\varepsilon\tilde{h}),$$ and since $\ln \lvert\sin x
\rvert $ is even, we have
\begin{equation}\label{t6.13_1}
  \lim_{\delta\rightarrow 0^+}\frac{\ud }{\ud \varepsilon}\biggr|_{\varepsilon=0}\Eh^\delta(h+\varepsilon\tilde{h})
  =\lim_{\delta\rightarrow 0^+}
  \int_0^L\bigg(\int_0^{x-\delta}+\int_{x+\delta}^L\bigg)\frac{-2\pi}{L}\cot\frac{\pi(x-y)}{L}h_y(y)  \tilde{h}(x) \ud y\ud x.
\end{equation}

Now we claim
$$\frac{\ud }{\ud \varepsilon}\biggr|_{\varepsilon=0}\lim_{\delta\rightarrow 0^+}\Eh^\delta(h+\varepsilon\tilde{h})
=\lim_{\delta\rightarrow 0^+}\frac{\ud }{\ud
  \varepsilon}\biggr|_{\varepsilon=0}\Eh^\delta(h+\varepsilon\tilde{h}).$$
Obviously, $\Eh^{\delta}(h+\varepsilon\tilde{h})$ is continuous
respect to $\delta$. It suffices to show that $\bigl .\frac{\ud }{\ud
  \varepsilon}\bigr|_{\varepsilon=0}\Eh^\delta(h+\varepsilon\tilde{h})$
is also continuous respect to $\delta$.  Hence, from \eqref{t6.13_1}, it suffices to prove
$$\lim_{\delta\rightarrow 0^+}\int_0^L\int_{x-\delta}^{x+\delta}
\frac{\pi}{L}\cot\frac{\pi(x-y)}{L}h_y(y)  \tilde{h}(x) \ud y \ud x =0.$$
Indeed
\begin{equation*}
\begin{aligned}
\int_0^L\int_{x-\delta}^{x+\delta}
\frac{\pi}{L}\cot\frac{\pi(x-y)}{L}h_y(y) \tilde{h}(x) \ud y  \ud x
=&\int_0^L -\ln\big\lvert\sin \frac{\pi(x-y)}{L}\big\rvert h_y(y) \biggr|_{y=x-\delta}^{x+\delta} \tilde{h}(x)\ud x \\
&+\int_0^L\int_{x-\delta}^{x+\delta}\ln\big\lvert\sin \frac{\pi(x-y)}{L}\big\rvert h_{yy}(y)\tilde{h}(x) \ud y  \ud x.
\end{aligned}
\end{equation*}
Notice that $h(x)\in C^2([0,L])$. Let $\delta\rightarrow 0.$ The first
term tends to zero by Taylor expansion, and the second term tends to
zero as the integrand is integrable.
\end{proof}

Note that the energy $\Eh$ we use here has a slightly different form compared to the one in \cite{Xiang2002}, denoted by $\bar{\Eh}(h)$, which reads in the periodic setting as
\begin{equation}
\bar{\Eh}(h)= \int_{0}^{L}\biggl(-\frac{\pi}{L}\big(h+\frac{x}{L}\big)H(h_x)-h_x\ln(-h_x)-\frac{h_x^3}{2}\biggr)\ud x.
\end{equation}
In fact, the two energy functionals only differ by a null Lagrangian, as we show below, so we prefer the more symmetric expression $E_h$.

\begin{lem}
Let
\begin{equation}\label{eq:W}
    W(h):=\frac{1}{L^2}\int_0^L\int_0^L\ln\big\lvert\sin  \frac{\pi(x-y)}{L}|h_y\ud x\ud y.
\end{equation}
Then we have
$$\Eh(h)=\bar{\Eh}(h)+W(h),$$
and
$$\frac{\delta \Eh}{\delta h}=\frac{\delta \bar{\Eh}}{\delta h}.$$
\end{lem}
\begin{proof}
First by the definition of the periodic Hilbert transform,
\begin{equation*}
  \bar{\Eh}(h)  = \int_{0}^{L}\biggl(-\frac{\pi}{L^2}\big(h+\frac{x}{L}\big)\PV\int_0^L \cot\frac{\pi(x-y)}{L}h_y \ud y-h_x\ln(-h_x)-\frac{h_x^3}{2}\biggr)\ud x.
\end{equation*}
Notice that
\begin{equation*}
\begin{aligned}
&\int_{0}^{L}\biggl(-\frac{\pi}{L^2}\big(h+\frac{x}{L}\big)\PV\int_0^L \cot\frac{\pi(x-y)}{L}h_y \ud y\biggr)\ud x\\
&=-\frac{1}{L}\int_0^L \biggl((h+\frac{x}{L})\ln\big\lvert\sin\frac{\pi(x-y)}{L}\big\rvert\biggr\rvert_0^L-PV\int_0^L\bigl(h_x+\frac{1}{L}\bigr)\ln\big\lvert\sin \frac{\pi(x-y)}{L}\big\rvert \ud x\biggr)h_y \ud y\\
&=\frac{1}{L}\int_0^L\int_0^L\ln\big\lvert\sin \frac{\pi(x-y)}{L}\big\rvert h_xh_y\ud x\ud y+\frac{1}{L^2}\int_0^L\int_0^L\ln\big\lvert\sin \frac{\pi(x-y)}{L}\big\rvert h_y\ud x\ud y,
\end{aligned}
\end{equation*}
where we have used that $h+\frac{x}{L}$ is $L$-periodic function. Therefore, for $W$ defined in \eqref{eq:W}, we get
\begin{equation*}
  \Eh(h)=\bar{\Eh}(h)+W(h).
\end{equation*}

Similar to the proof of Lemma \ref{l3.0}, we can see
\begin{align*}
\langle\frac{\delta W}{\delta h},\tilde{h}\rangle&=\frac{1}{L^2}\int_0^L\int_0^L\ln\big\lvert\sin  \frac{\pi(x-y)}{L}|\ud x \tilde{h}_y \ud y,\\
&=\frac{1}{L^2}\int_0^L \tilde{h} \ln\big\lvert\sin\frac{\pi(x-y)}{L}\big\rvert\biggr\rvert_0^L \ud x-\int_0^L \PV\int_0^L \frac{2\pi}{L^2}\cot\frac{\pi(x-y)}{L} \ud x \tilde{h}(y) \ud y\\
&=0.
\end{align*}
Hence $W(h)$ is a null lagrangian.
\end{proof}

\subsection{Equation for step location function $\phi$}

Consider the step location function $\phi$, which defined in \eqref{phi} as the inverse function of $h$. From the definition, we have
\begin{equation}\label{fan}
\begin{aligned}
  \phi_t=-\frac{h_t}{h_x},\qquad 1=h_x \phi_{\alpha},\qquad
  h_{xx}=-\frac{\phi_{\alpha\alpha}}{\phi_\alpha^3}.
\end{aligned}
\end{equation}
Then changing variable from $h$ to $\phi$ in \eqref{h}, we have
\begin{equation}\label{phi2}
  \phi_t=-\phi_\alpha \mu_{xx}=-\partial_\alpha\biggl(\frac{1}{\phi_\alpha} \mu_{\alpha}\biggr),
\end{equation}
due to \eqref{fan} and the chain rule
$\mu_x=\mu_{\alpha}\frac{1}{\phi_\alpha}$. Note that this immediately implies that $\int_0^1 \phi \ud \alpha$ is a constant of motion.

The equation of $\phi$ \eqref{phi2} also has a variational structure. To this end, let us rewrite the energy in terms of $\phi$ such that $\Ephi(\phi)=\Eh(h)$:
\begin{equation}
\Ephi(\phi)=  \int_{0}^{1}\biggl(\frac{1}{L}\int_0^1\ln\bigl\lvert\sin\frac{\pi(\phi(\alpha)-\phi(\beta))}{L}\bigr\rvert \ud\beta-{\ln(-\phi_\alpha)}+\frac{1}{2\phi_\alpha^2}\biggr) \ud\alpha.
\end{equation}
We will show that
\begin{equation}\label{phi3}
  \phi_t=-\phi_\alpha \mu_{xx}=-\partial_\alpha\biggl(\frac{1}{\phi_\alpha}\Big(\frac{\delta \Ephi}{\delta \phi}\Big)_\alpha\biggr).
\end{equation}

Similar to the proof of Lemma \ref{l3.0}, let us first calculate
$\frac{\delta \Ephi^0}{\delta \phi}$,
where
$$\Ephi^0(\phi):=\int_0^1\int_0^1\ln\big\lvert\sin
\frac{\pi(\phi(\alpha)-\phi(\beta))}{L}\big\rvert \ud\alpha \ud
\beta.$$

\begin{lem}\label{l3.1}
Assume $h(x)\in C^2([0,L])$ and there exists a constant $C>0$ such that $|h_x|\geq C$. We have
$$\frac{\delta \Ephi^0}{\delta \phi}=\PV\int_0^1\frac{2\pi}{L}\cot\frac{\pi(\phi(\alpha)-\phi(\beta))}{L} \ud \beta .$$
\end{lem}
\begin{proof}
First denote
$$\Ephi^\delta(\phi):=\int_0^1\bigg(\int_0^{\beta-\delta}+\int_{\beta+\delta}^1\bigg)\ln\big\lvert\sin \frac{\pi(\phi(\alpha)-\phi(\beta))}{L}\big\rvert \ud\alpha \ud \beta .$$
It is obvious to see that
$$\frac{\ud }{\ud \varepsilon}\biggr|_{\varepsilon=0}\Ephi^0(\phi+\varepsilon\tilde{\phi})
=\frac{\ud }{\ud \varepsilon}\biggr|_{\varepsilon=0}\lim_{\delta\rightarrow 0}\Ephi^\delta(\phi+\varepsilon\tilde{\phi}),$$
and
$$\frac{\ud }{\ud \varepsilon}\biggr|_{\varepsilon=0}\Ephi^\delta(\phi+\varepsilon\tilde{\phi})
=\int_0^1\bigg(\int_0^{\beta-\delta}+\int_{\beta+\delta}^1\bigg)\frac{\pi}{L}\cot\frac{\pi(\phi(\alpha)-\phi(\beta))}{L}(\tilde{\phi}(\alpha)-\tilde{\phi}(\beta))\ud\alpha \ud \beta .$$

Now we claim
$$\frac{\ud }{\ud \varepsilon}\biggr|_{\varepsilon=0}\lim_{\delta\rightarrow 0^+}\Ephi^\delta(\phi+\varepsilon\tilde{\phi})
=\lim_{\delta\rightarrow 0^+}\biggr
.\frac{\ud }{\ud \varepsilon}\biggr|_{\varepsilon=0}\Ephi^\delta(\phi+\varepsilon\tilde{\phi}).$$
Obviously, $\Ephi^{\delta}(\phi+\varepsilon\tilde{\phi})$ is continuous respect to $\delta$. It is sufficient to proof $\bigl
.\frac{\ud }{\ud \varepsilon}\bigr|_{\varepsilon=0}\Ephi^\delta(\phi+\varepsilon\tilde{\phi})$ is also continuous respect to $\delta$.
In fact, since $\cot x$ is odd,
$$\frac{\ud }{\ud \varepsilon}\biggr|_{\varepsilon=0}\Ephi^\delta(\phi+\varepsilon\tilde{\phi})=2\int_0^1\bigg(\int_0^{\beta-\delta}+\int_{\beta+\delta}^1\bigg)
\frac{\pi}{L}\cot\frac{\pi(\phi(\alpha)-\phi(\beta))}{L}\tilde{\phi}(\alpha)\ud\alpha \ud \beta .$$
Hence, it is sufficient to proof
$$\lim_{\delta\rightarrow 0^+}\int_0^1\int_{\beta-\delta}^{\beta+\delta}
\frac{\pi}{L}\cot\frac{\pi(\phi(\alpha)-\phi(\beta))}{L}\tilde{\phi}(\alpha)\ud\alpha \ud \beta =0.$$
In fact,
 \begin{equation*}
\begin{aligned}
&\int_0^1\int_{\beta-\delta}^{\beta+\delta}
\frac{\pi}{L}\cot\frac{\pi(\phi(\alpha)-\phi(\beta))}{L}\tilde{\phi}(\alpha)\ud\alpha \ud \beta \\
=&\int_0^1 \frac{\tilde{\phi}(\alpha)}{\phi_{\alpha}(\alpha)}\ln\big\lvert\sin \frac{\pi(\phi(\alpha)-\phi(\beta))}{L}\big\rvert \biggr|_{\alpha=\beta-\delta}^{\beta+\delta} \ud \beta \\
&-\int_0^1\int_{\beta-\delta}^{\beta+\delta}\ln\big\lvert\sin \frac{\pi(\phi(\alpha)-\phi(\beta))}{L}\big\rvert \biggl(\frac{\tilde{\phi}(\alpha)}{\phi_{\alpha}(\alpha)}\biggr)_{\alpha}\ud\alpha \ud \beta.
\end{aligned}
\end{equation*}
As $\delta\rightarrow 0,$ the first term tends to zero by Taylor
expansion.  $\bigl
\lvert\bigl(\frac{\tilde{\phi}(\alpha)}{\phi_{\alpha}(\alpha)}\bigr)_{\alpha}\bigr\rvert$
is bounded since $h(x)\in C^2([0,L])$ and $|h_x|\geq C>0$, so the
second term tends to zero as the integrand is integrable.
\end{proof}

Hence we have
\begin{equation}\label{Frac_phi}
\frac{\delta \Ephi}{\delta \phi}=\frac{2\pi}{L^2}\PV\int_0^1 \cot\frac{\pi(\phi(\alpha)-\phi(\beta))}{L}\ud \beta -\frac{\phi_{\alpha\alpha}}{\phi_{\alpha}^2}-3\frac{\phi_{\alpha\alpha}}{\phi_{\alpha}^4}.
\end{equation}


\smallskip

It remains to show that $\mu = \frac{\delta \Ephi}{\delta \phi}$,
\textit{i.e.},
$\frac{\delta \Ephi}{\delta \phi} = \frac{\delta \Eh}{\delta h}$.  For
$\tilde{\phi},\,\tilde{h}$ satisfying
$$\alpha=(h +\varepsilon \tilde{h})\circ(\phi+\varepsilon \tilde{\phi}),$$
Taylor expansion shows that
$$0=h_x \tilde{\phi}+\tilde{h}.$$
Thus by \eqref{fan},  we have
\begin{equation}\label{tilde}
\begin{aligned}
\tilde{\phi}=-\phi_{\alpha}\tilde{h},\\
\Ephi(\phi+\varepsilon \tilde{\phi})=\Eh(h+\varepsilon \tilde{h}).
\end{aligned}
\end{equation}
Hence
\begin{equation}\label{EF}
\begin{aligned}
&\frac{\ud }{\ud \varepsilon}\biggr|_{\varepsilon=0}\Ephi(\phi+\varepsilon \tilde{\phi})=D_\phi \Ephi \cdot \tilde{\phi}\\
=&\frac{\ud }{\ud \varepsilon}\biggr|_{\varepsilon=0}\Eh(h+\varepsilon \tilde{h})=D_h \Eh \cdot \tilde{h},
\end{aligned}
\end{equation}
where $D_h\Eh:L^2(\mathbb{R})\rightarrow L^2(\mathbb{R})$ is the Fr\'echet differential, i.e. $D_h \Eh \cdot \tilde{h}$ is the dual pair which means the first order variation of $\Eh$ at $h$ along the direction of $\tilde{h}$.

By Riesz representation theorem, there exists $\nabla_h \Eh\in L^2([0,L],\ud x)$, such that
$$ D_h \Eh \cdot \tilde{h}=\int_{0}^{L}\nabla_h \Eh \tilde{h} \ud x,$$
where $\nabla_{h} \Eh$ is gradient of $\Eh(h)$ in $L^2([0,L],\ud x)$, which is just what we denoted as  $\frac{\delta \Eh}{\delta h}$.

Similarly, there exists $\nabla_\phi \Ephi\in L^2([0,1],\vert\phi_\alpha\vert \ud\alpha)$, such that
$$ D_\phi \Ephi \cdot \tilde{\phi}=\int_{0}^{1}\nabla_\phi \Ephi \tilde{\phi} \vert\phi_\alpha\vert \ud\alpha=\int_{0}^{1}-\nabla_\phi \Ephi \tilde{\phi} \phi_\alpha \ud\alpha. $$
where $\nabla_{\phi} \Ephi$ is gradient of $\Ephi(\phi)$ in $L^2([0,1],\vert\phi_\alpha\vert \ud\alpha)$.

Combining \eqref{tilde} and \eqref{EF}, we get
$$\nabla_\phi \Ephi=-\frac{1}{\phi_\alpha}\nabla_h \Eh\circ \phi.$$
Again we define $\frac{\delta \Ephi}{\delta \phi}$ as gradient of $\Ephi(\phi)$ in $L^2([0,1],\ud\alpha)$. Noticing \eqref{tilde}, we have
\begin{equation*}
\begin{aligned}
&\biggl
.\frac{\ud }{\ud \varepsilon}\biggr|_{\varepsilon=0}\Ephi(\phi+\varepsilon \tilde{\phi})=\int_{0}^{1}\frac{\delta \Ephi}{\delta \phi}\tilde{\phi}\ud \alpha\\
=&\frac{\ud }{\ud \varepsilon}\biggr|_{\varepsilon=0}\Eh(h+\varepsilon \tilde{h})=\int_{0}^{L}\nabla_h \Eh\tilde{h}\ud x\\
=&\int_{0}^{1}\frac{\delta \Eh}{\delta h}\tilde{\phi}\ud\alpha.
\end{aligned}
\end{equation*}
Hence
$$\frac{\delta \Eh}{\delta h}\circ \phi=\frac{\delta \Ephi}{\delta \phi}\in L^2([0,1],\ud\alpha),$$
and
$$\mu=\frac{\delta \Eh}{\delta h}\circ \phi=\nabla_h \Eh\circ \phi=-\phi_\alpha \nabla_{\phi} \Ephi=\frac{\delta \Ephi}{\delta \phi}.$$

Therefore, we conclude that \eqref{phi2} is equivalent with \eqref{phi3}. Moreover, we obtain energy identity for \eqref{phi3} as
\begin{equation}\label{E3.8}
\frac{\ud \Ephi}{\ud t}=\int_{0}^{1}\frac{\delta \Ephi}{\delta \phi} \phi_t d \alpha=\int_{0}^{1}\frac{1}{\phi_\alpha}\biggl(\Big(\frac{\delta \Ephi}{\delta \phi}\Big)_\alpha\biggr)^2 \ud \alpha.
\end{equation}

\subsection{Equation for step density $\rho$}

Now consider the step density $\rho$. From the definition,
rewriting the energy in terms of $\rho$, we obtain
\begin{equation}\label{30Erho}
\Erho(\rho):=\int_{0}^{L}\biggl
(\frac{1}{L}\int_{0}^{L}\ln\big\lvert\sin  (\frac{\pi}{L}(x-y))\big\rvert\rho(x)\rho(y) \ud y +\rho(x)\ln\rho(x)+\frac{\rho(x)^3}{2}\biggr)\ud x,
\end{equation}
$$\frac{\delta \Erho}{\delta\rho}=\int_{0}^{L}\frac{2}{L}\ln\big\lvert\sin  (\frac{\pi}{L}(x-y))\big\rvert\rho(y) \ud y +\ln\rho(x)+1+\frac{3}{2}\rho(x)^2,$$
and
\begin{equation}\label{temp3.10}
\biggl
(\frac{\delta \Erho}{\delta \rho}\biggr)_x= \PV\int_0^L \frac{2\pi}{L^2}\cot\frac{\pi(x-y)}{L}\rho(y)\ud y+\frac{\rho_x}{\rho}+3\rho_x \rho=\mu.
\end{equation}

Similar to the proof of Lemma \ref{l3.0}, we can define
$$\PV\int_0^L  \cot\frac{\pi(x-y)}{L}\rho(y)\ud y=\lim_{\delta\rightarrow 0^+}\Big(\int_0^{x-\delta}+\int_{x+\delta}^L\Big) \cot\frac{\pi(x-y)}{L}\rho(y)\ud y.$$
Then
\begin{align*}
&\frac{\ud}{\ud x}\lim_{\delta\rightarrow 0^+}\Big(\int_0^{x-\delta}+\int_{x+\delta}^L\Big)\ln \big\lvert\sin\frac{\pi(x-y)}{L}\big\rvert\rho(y) \ud y\\
=& \lim_{\delta\rightarrow 0^+}\frac{\ud}{\ud x}\Big(\int_0^{x-\delta}+\int_{x+\delta}^L\Big)\ln \big\lvert\sin\frac{\pi(x-y)}{L}\big\rvert\rho(y) \ud y.
\end{align*}

Hence we also obtain a variational structure for $\rho$ and \eqref{h} becomes
\begin{equation}\label{rhoeq}
\rho_t=-\mu_{xxx}=-\biggl(\frac{\delta \Erho}{\delta \rho}\biggr)_{xxxx}.
\end{equation}
This also shows that $\int_0^L \rho \ud x$ is a constant of motion.

\subsection{Equation for $u$}

Finally, from definition of $u$, the energy can be rewritten in terms of $u$ as
\begin{equation}\label{30Eu}
\begin{aligned}
\Eu(u)&=\int_{0}^{L}\left(\frac{1}{L}\int_{0}^{L}\ln|\sin(\frac{\pi}{L}(x-y))|(u_{xx}+b)(u_{yy}+b) dy +(u_{xx}+b)\ln (u_{xx}+b)+\frac{(u_{xx}+b)^3}{2}\right)dx,\\
\frac{\delta \Eu}{\delta u}&=\frac{2\pi}{L} H(u_{xx})_x+\Bigl(\ln{(u_{xx}+b)}+\frac{3}{2}(u_{xx}+b)^2+1\Bigr)_{xx}=\mu_x.
\end{aligned}
\end{equation}
Hence we also obtain a variational structure for $u$ and \eqref{h} becomes
\begin{equation}\label{Du1}
u_t=-\frac{\delta \Eu}{\delta u}.
\end{equation}

%

\subsection{Equivalence of the formulations}

We end this section with the rigorous justification of the equivalence of the above formulations.

Recall the notations for $W^{k,p}_{\wps}(I)$, $W^{k,p}_{\per_0}(I)$ in \eqref{626_1} and \eqref{626_2}. If $k<0$ and $\frac{1}{p}+\frac{1}{q}=1$,
$W^{k,p}$ is the dual of $W^{-k,q}$.
Denote
$$\Phi(\xi):=\left\{
              \begin{array}{ll}
                \xi \ln \xi+\frac{\xi^3}{2}, &\quad \xi>0, \\
                0, &\quad \xi=0, \\
                +\infty, &\quad \xi<0,
              \end{array}
            \right.
$$
and
$$\Phib(\xi):=\Phi(\xi+b).$$

By the definition \eqref{30Erho}, we have
\begin{equation}\label{temp4.2}
\Erho(\rho)=\int_{0}^{L}\biggl(\frac{1}{L}\int_{0}^{L}\ln\big\lvert\sin  (\frac{\pi}{L}(x-y))|\rho(x)\rho(y) \ud y +\Phi(\rho)\biggr)\ud x.
\end{equation}
By \eqref{30Eu}, we have
$$\Eu(u)=\int_{0}^{L}\biggl(\frac{1}{L}\int_{0}^{L}\ln\big\lvert\sin  (\frac{\pi}{L}(x-y))|(u_{xx}+b)(u_{yy}+b) \ud y +\Phi_b(u_{xx})\biggr)\ud x.$$

 Since
$$\frac{\delta \Eu(u)}{\delta u}= \lpi H(u_{xx})_x+(\Phibp(u_{xx}))_{xx},$$
the equation \eqref{Du1} can be recast as
\begin{equation}\label{ua}
\begin{aligned}
&u_t+\lpi H(u_{xx})_x+(\Phibp(u_{xx}))_{xx}=0.
\end{aligned}
\end{equation}

In order to study the problem \eqref{5} in periodic and mean value zero set up, we establish first, similar to \cite{Leoni2014}, that
\begin{prop}\label{h-u}
For any integer $m\geq 1$, any $T>0$ and some constant $\beta<0$,  the following condition are equivalent:

(a) There exists $h\in L^\infty([0,T]; W^{m,3}_{\wps}(I))$ with $h_t\in L^{\infty}([0,T];W^{m-4,3/2}_{\per}(I))$ a solution of \eqref{5} satisfying
$$h_x(x,t)\leq\beta<0 \quad \ale x \in \mathbb{R},\, t\in[0,T].$$

(b) Set $b:=\frac{1}{L}>0$. There exists $u\in L^\infty([0,T]; W^{m+1,3}_{\per_0}(I))$ with $u_t\in L^{\infty}([0,T];W^{m-3,3/2}_{\per_0}(I))$ a solution of \eqref{ua}
 satisfying
$$u_{xx}(x,t)+b\geq -\beta>0 \quad \ale x \in \mathbb{R},\, t\in[0,T].$$

(c) There exists $\rho\in L^\infty([0,T]; {W}^{m-1,3}_{\per}(I))$ with $\rho_t\in L^{\infty}([0,T];({W}^{m-5,3/2}_{\per}(I)))$ a solution of
\eqref{rhoeq}
 satisfying
$$\rho(x,t)\geq -\beta>0 \quad \ale x \in \mathbb{R},\, t\in[0,T],$$
and
$$\int_0^L \rho(x,t)\ud x=1.$$
\end{prop}
\begin{proof}
Step 1. For (a)$\Rightarrow$(c), we simply take
\begin{equation}\label{rhotemp}
\rho(t,x):=-h_x(t,x)=u_{xx}(t,x)+b
\end{equation}
and then \eqref{temp3.10} shows that $\rho$ satisfies (c).

For (c)$\Rightarrow$(a), we take
$$h(x,t)=-\int_0^x \rho(s,t)ds+k_2(t),$$
with
$$k_2(t)=\frac{1}{L}\int_0^L \int_0^x \rho(y,t)\ud y \ud x.$$
Then $h_x=-\rho$ and $h\in L^\infty([0,T]; W^{m,3}_{\wps}(I)),$ with mean value zero.

 Noticing \eqref{temp3.10} again,
 we have
 $$h_{xt}=-\rho_t=\biggl(\frac{\delta \Erho}{\delta \rho}\biggr)_{xxxx}=\biggl(\frac{\delta \Eh}{\delta h}\biggr)_{xxx},$$
 in distribution sense.
 Integrating from 0 to $x$, for \ale $t\in [0,T]$, there exists a constant $c(t)$ such that
 $$h_t=\biggl(\frac{\delta \Eh}{\delta h}\biggr)_{xx}+c(t).$$
 That is, for any test function $\varphi\in W^{3,3}_{\per}(I),$ we have
 $$\frac{\ud }{\ud t}\langle h,\varphi{}\rangle=\langle\frac{\delta \Eh}{\delta h},\varphi_{xx}{}\rangle+\langle c(t),\varphi{}\rangle.$$
Taking $\varphi=1$, we get $c(t)=0,\text{ for }\ale t\in[0,T].$ Hence $h$ is the solution of \eqref{5}.

Step 2.
For (a)$\Rightarrow$(b), we take
$$h^T(x,t)=h(x,t)+bx,$$
with $b=\frac{1}{L}.$ From \eqref{2} and \eqref{5}, we know $h^T$ is $L-$periodic function respect to $x$.

Denote
$$k_0=\frac{1}{L}\int_0^L h^T(s,0)ds,$$
$$k_1(t)=\frac{1}{L}\int_0^L \int_0^x h^T(y,t)\ud y \ud x-k_0\frac{L}{2}.$$
Set
\begin{equation}\label{620_2}
u(x,t)=\int_0^x\biggl(-h^T(y,t)+k_0 \biggr)\ud y+k_1(t).
\end{equation}
We know $u$ is $L-$periodic function with mean value zero.
To prove such $u$ satisfies \eqref{ua}, we can proceed just the same as Step 1.

Note we also have
\begin{equation}\label{star7}
u_x=-h-bx+k_0,
\end{equation}
\begin{equation}\label{star8}
u_{xx}=-h_x-b.
\end{equation}

For (b)$\Rightarrow$(a),
we simply take
\begin{equation}
h=-u_x-bx.
\end{equation}
Then \eqref{30Eu} and \eqref{Du1} show that $h$ satisfies (b).
\end{proof}

\begin{prop}\label{h-phi}
For any integer $m\geq 2$, the following condition are equivalent:

(i) There exists $h\in L^\infty([0,T]; W^{1,\infty}_{\wps}(I)\cap W^{m,2}(I))$ with $h_t\in L^{\infty}([0,T];W^{-3,\infty}_{\per}(I))$ a solution of \eqref{5} satisfying
\begin{equation}\label{con1}
h_x(x,t)\leq \beta_1 <0 \quad \ale  x \in \mathbb{R},\, t\in[0,T],
\end{equation}
for some $\beta_1<0.$

(ii) There exists $\phi\in L^\infty([0,T]; W^{1,\infty}_{\wps}([0,1])\cap W^{m,2}([0,1]))$ with $\phi_t\in L^{\infty}([0,T];W^{-3,\infty}_{\per}([0,1]))$ a solution of \eqref{phi3} satisfying
\begin{equation}\label{con2}
\phi_\alpha(\alpha,t)\leq \beta_2<0 \quad \ale \alpha \in \mathbb{R},\, t\in[0,T],
\end{equation}
for some $\beta_2<0.$
\end{prop}
\begin{proof} Notice condition \eqref{con1}, \eqref{con2}. By inverse function theorem, $h$ and $\phi$ are inverse functions of each other. Noticing \eqref{phi} and \eqref{fan}, $h\in L^\infty([0,T]; W^{1,\infty}_{\wps}(I))$ with condition \eqref{con1} implies that $\phi\in L^\infty([0,T]; W^{1,\infty}_{\wps}([0,1]))$ with condition \eqref{con2}.

From the differentiation of inverse function, we also know
$$\phi^{(m)}\leq C(\beta_1)(h^{(m)}+\sum_{ {0\leq \alpha_i\leq m-1}}h^{(\alpha_1)}h^{(\alpha_2)}\cdots h^{(\alpha_m)}).$$
Since $W^{m,2}\hookrightarrow W^{(m-1),\infty},$
we have
$$\int_0^L  \lvert\phi^{(m)}\rvert ^2\ud\alpha \leq C(\beta_1)(\|h\|^2_{W^{m,2}}+\|h\|^m_{W^{m,2}}).$$
Hence $h\in L^\infty([0,T]; W^{m,2}(I))$ with condition \eqref{con1} implies that $\phi\in L^\infty([0,T]; W^{m,2}([0,1]))$ with condition \eqref{con2}.
Vice versa.
\end{proof}

\section{Local strong solution and proof of Theorem \ref{local_h}}\label{sec4}
We continue studying the properties of the continuum PDE.
From now on, denote
$$ \varphi^{(n)}(x)=\frac{\ud ^n}{\ud  x^n} \varphi(x),$$
and $c$ as a generic constant whose value may change from line to line.
We first establish the existence and uniqueness of the local strong solution to \eqref{ua}.
\begin{thm}\label{local}
Assume $u^0\in W^{m,2}_{\per_0}(I),$ $u_{xx}^0+b\geq \eta$, where $\eta$ is a positive constant, $m\in \mathbb{Z},\, m\geq 7$. Then there exists time $T_{m}$ depending on $\eta,\,\|u^0\|_{W^{m,2}_{\per_0}}$
such that
$$u\in L^{\infty}([0,T_m];W^{m,2}_{\per_0}(I))\cap L^2([0,T_m];W^{m+2,2}_{\per_0}(I))\cap C([0,T_m];W^{m-4,2}_{\per_0}(I)),$$
$$u_t\in L^{\infty}([0,T_m];W^{m-4,2}_{\per_0}(I))\cap L^2([0,T_m];L^{2}_{\per_0}(I))$$
is the unique strong solution of \eqref{ua} with initial data $u^0$, and $u$ satisfies
$$u_{xx}+b\geq\frac{\eta}{2},\quad \ale  t\in[0,T_m],\,x\in [0,L].$$
\end{thm}

\begin{proof}
We first make the {\it{a-priori}} assumption
\begin{equation}\label{a-p1}
\min_{x\in I} (u_{xx}+b)\geq \frac{\eta}{2}>0,\quad \ale  t\in[0,T_m],
\end{equation}
in which $T_m$ will be determined later. We will prove the existence of local strong solution under \eqref{a-p1} in step 1,2, then justify \eqref{a-p1} in step3.

Let $J_{\delta}$ be the standard $C_c^{\infty}(I)$ mollifier. Denote $\bar{u}^{\delta}=J_{\delta}*u^{\delta}$.

Define $\Eud(u):=\Eu(J_\delta * u)$. Then
$$\frac{\delta \Eud( u^{\delta})}{\delta u^{\delta}}=J_\delta * \biggl.\frac{\delta \Eu(u)}{\delta u}\biggr|_{\bar{u}^{\delta}}.$$
We study problem
\begin{equation}\label{XXX}
\biggl\{
  \begin{array}{ll}
 u_t^{\delta}=-\frac{\delta \Eud(u^{\delta})}{\delta u^{\delta}},\\
    u^{\delta}(0)=J_\delta *u^{0},
  \end{array}
\biggr.
\end{equation}
which is
\begin{equation}\label{reg1}
\biggl\{
  \begin{array}{ll}
    u^{\delta}_t=(J_{\delta}*(-\lpi H(\bar{u}_{xx}^\delta)))_x-(J_{\delta}*\Phibp(\bar{u}_{xx}^\delta))_{xx}, \\
    u^{\delta}(0)=J_{\delta}*u^0.
  \end{array}
\biggr.
\end{equation}

Step 1. We devote to obtain some \textit{a-priori} estimates, which will be used to prove the convergence of $u^\delta$ in \eqref{XXX}.

Taking $u$ as a test function in \eqref{ua} gives
$$\int_0^L u_t u \ud x=\int_0^L \lpi H(u_{xx})u_x-(\ln (u_{xx}+b)+
\frac{3}{2}(u_{xx}+b)^2)u_{xx} \ud x.$$
Notice that
$$\int_0^L H(u_{xx})u_x \ud x \leq \int_0^L \frac{3}{4}u_{xx}^{ 2}+2u^{ 2} \ud x \leq \int_0^L \frac{1}{8}u^{ 3}_{xx}+2 u^{ 2} \ud x +C(L), $$
and that
$$\int_0^L \ln( u_{xx}+b) u_{xx} \ud x \leq C(\eta,L)+\frac{1}{8} \int_0^L u^{ 3}_{xx} \ud x,$$
due to \eqref{a-p1}. We obtain
$$\frac{\ud }{\ud t}\int_0^L u^{2}\ud x+\int_0^L u^{ 3}_{xx} \ud x \leq c\int_0^L u^{2}\ud x +C(\eta,L).$$
Then for some $T_1>0$, Gr\"onwall's inequality implies that
$$\|u\|_{L^{\infty}([0,T_1];L^2(I))}\leq C(\eta,L,\|u^0\|_{W_{\per_0}^{m,2}},T_1),$$
\begin{equation}
\|u_{xx}\|_{L^2([0,T_1]; L^3(I))}\leq C(\eta,L,\|u^0\|_{W_{\per_0}^{m,2}},T_1).
\end{equation}
Here and the following, $C(\eta,L,\|u^0\|_{W_{\per_0}^{m,2}},T_1)$ is a constant depending only on $\eta,\,L,\,\|u^0\|_{W_{\per_0}^{m,2}} $ and $T_1$.

Recall \eqref{rhotemp}. We use $\rho=u_{xx}+b$ from now.

Since
$$\frac{\ud \Eu(u)}{\ud t}+\int_0^L  \biggl(\frac{\delta \Eu(u)}{\delta u}\biggr)^2\ud x=0,$$
we have
\begin{equation}\label{Eau}
\Eu(u)\leq \Eu(u_0)<+\infty.
\end{equation}

Also notice
\begin{equation}\label{rho1}
\begin{aligned}
&|\int_0^L  \int_0^L  \ln\big\lvert\sin  \frac{\pi}{L}(x-y)\big\rvert \rho(x)\rho(y)\ud x\ud y|\\
\leq& \biggl(\int_0^L \int_0^L  \ln^2\big\lvert\sin  \frac{\pi}{L}(x-y)\big\rvert \ud x\ud y\biggr)^{\frac{1}{2}}\int_0^L \rho^2(x)\ud x\\
\leq & \frac{1}{8}\int_0^L  \rho^3\ud x +C(L),
\end{aligned}
\end{equation}
and
\begin{equation}\label{rho2}
\begin{aligned}
&|\int_0^L  \rho \ln \rho \ud x|
\leq & \frac{1}{8}\int_0^L  \rho^3\ud x +C(\eta,L).
\end{aligned}
\end{equation}
These, together with \eqref{Eau}, give that
\begin{equation}\label{uxxc}
\frac{1}{4}\sup_{0\leq t\leq T_1}\int_0^L  \rho^3 \ud x <\Erho(0)+C(\eta,L).
\end{equation}

Now we devote to get a higher-order priori estimate for $m\geq 4.$

Divide $m$ times in equation \eqref{ua} and then take $u^{(m)}$ as a test function, which implies that
\begin{equation}\label{Hm}
\frac{\ud }{\ud t}\|u\|_{\dot{W}^{m,2}}=\int_0^L  -\lpi H(\rho)^{(m+1)}u^{(m)}-f(\rho)^{(m+2)}u^{(m)} \ud x,
\end{equation}
where
$$f(\rho)=\Phi'(\rho)=\ln \rho +1+\frac{3}{2} \rho^2.$$
For the first term in \eqref{Hm}, we have
\begin{equation}\label{TT1}
\begin{aligned}
|\int_0^L  -H(\rho)^{(m+1)}u^{(m)} \ud x|=&|\int_0^L  -H(\rho)^{(m)}\rho^{(m-1)}\ud x|\\
\leq & \frac{1}{8}\int_0^L  \rho^{(m)2}\ud x +2\int_0^L  \rho^{(m-1)2}\ud x\\
\leq & \frac{1}{4} \int_0^L \rho^{(m)2}\ud x +c\int_0^L  \rho^{(m-2)2}\ud x.
\end{aligned}
\end{equation}

For the second term in \eqref{Hm}, we have
\begin{equation}\label{TT2}
\begin{aligned}
\int_0^L  -f(\rho)^{(m+2)}u^{(m)} \ud x=&\int_0^L  -f(\rho)^{(m)}\rho^{(m)}\ud x\\
= & \int_0^L  -(f'(\rho)\rho_x)^{(m-1)}\rho^{(m)}\ud x\\
= & \int_0^L  -f'(\rho)\rho^{(m)2}\ud x+\int_0^L  \sum_{k=0}^{m-2} C_k f'(\rho)^{(m-1-k)}\rho_x^{(k)}\rho^{(m)} \ud x.
\end{aligned}
\end{equation}
Note that
$$f'(\rho)=3\rho+\frac{1}{\rho}\geq 2\sqrt{3}, \text{ for }\rho>0,$$
so the first term on the right hand of \eqref{TT2} is strictly negative. We will use it to control the other terms later.

Now we carefully estimate the last term in \eqref{TT2}.
Denote
\begin{equation*}
\begin{aligned}
M_1:=&\int_0^L  \sum_{k=0}^{m-2} C_k f'(\rho)^{(m-1-k)}\rho_x^{(k)}\rho^{(m)} \ud x\\
\leq & \|\rho^{(m)}\|_{L^2} \biggl[\sum_{k=0}^{m-2} C_k \|f'(\rho)^{(m-1-k)}\rho_x^{(k)}\|_{L^2}\biggr].
\end{aligned}
\end{equation*}

First the chain rule gives
$$f'(\rho)^{(m-1-k)}=\sum_{\beta_1+\beta_2+\cdots+\beta_\mu=m-1-k} C_\beta \rho^{(\beta_1)}\rho^{(\beta_2)}\cdots\rho^{(\beta_\mu)}f^{(\mu+1)}(\rho).$$
Due to \eqref{a-p1}, we know
$$f^{(\mu+1)}(\rho)\leq \frac{C_\mu}{\rho^{\mu+1}}\leq \frac{C_\mu}{\eta^{\mu+1}}, \text{ for }\mu\geq 1.$$
Also noticing that
$$\|\rho^{(m-3)}\|_{L^\infty}\leq c\|\rho\|_{W^{m-2,2}},$$
we have
$$\|f'(\rho)^{(m-1-k)}\|_{L^4}\leq C(\eta,m)\|\rho\|_{W^{m-2,2}}^{m-1}, \text{ for }2\leq k \leq m-2,$$
$$\|f'(\rho)^{(m-2)}\|_{L^4}\leq C(\eta,m)(\|\rho\|_{W^{m-2,2}}^{m-1}+\|\rho^{(m-2)}\|_{L^4}), \text{  for }k=1,$$
and
$$\|f'(\rho)^{(m-1)}\|_{L^4}\leq C(\eta,m)(\|\rho\|_{W^{m-2,2}}^{m-1}+\|\rho^{(m-2)}\|_{L^4}+\|\rho^{(m-1)}\|_{L^4}), \text{ for }k=0.$$

Second by interpolating, we know
\begin{equation}\label{st1}
\|\rho^{(m-2)}\|_{L^4}\leq c \|\rho^{(m-2)}\|_{L^2}^{\frac{7}{8}}\|\rho^{(m)}\|_{L^2}^{\frac{1}{8}},
\end{equation}
\begin{equation}\label{st2}
\|\rho^{(m-1)}\|_{L^4}\leq c \|\rho^{(m-2)}\|_{L^2}^{\frac{3}{8}}\|\rho^{(m)}\|_{L^2}^{\frac{5}{8}},
\end{equation}
and for $\mu< m-2,$
\begin{equation}\label{st3}
\|\rho^{(\mu)}\|_{L^4}\leq c\|\rho^{(m-2)}\|_{L^4}+c\|\rho\|_{L^4} \leq c \|\rho^{(m-2)}\|_{L^2}^{\frac{7}{8}}\|\rho^{(m)}\|_{L^2}^{\frac{1}{8}}+c\|\rho\|_{W^{m-2,2}}.
\end{equation}

Thus \eqref{st1}, \eqref{st2} and \eqref{st3} show that
\begin{equation}\label{st4}
\begin{aligned}
&\sum_{k=0}^{m-2} C_k \|f'(\rho)^{(m-1-k)}\rho_x^{(k)}\|_{L^2}\\
\leq &\sum_{k=0}^{m-2} C_k \|f'(\rho)^{(m-1-k)}\|_{L^4}\|\rho_x^{(k)}\|_{L^4}\\
\leq & c \|f'(\rho)^{(m-2)}\|_{L^4}\|\rho_{xx}\|_{L^4}+\sum_{k=1}^{m-2} C(k,\eta,m)\|\rho\|_{W^{m-2,2}}^{m-1}(\|\rho^{(m-2)}\|_{L^2}^{\frac{7}{8}}\|\rho^{(m)}\|_{L^2}^{\frac{1}{8}}+c\|\rho\|_{W^{m-2,2}})\\
&+C(\eta,m)\|\rho\|_{W^{m-2,2}}^{m-1}\|\rho^{(m-2)}\|_{L^2}^{\frac{3}{8}}\|\rho^{(m)}\|_{L^2}^{\frac{5}{8}}
\end{aligned}
\end{equation}
For the first term, we have
\begin{equation}\label{st5}
\begin{aligned}
&\|f'(\rho)^{(m-2)}\|_{L^4}\|\rho_{xx}\|_{L^4}\\
\leq & C(\eta,m)(\|\rho\|_{W^{m-2,2}}^{m-1}+\|\rho^{(m-2)}\|_{L^4})(\|\rho^{(m-2)}\|_{L^4}+\|\rho\|_{W^{m-2,2}})\\
\leq &  C(\eta,m)\Big[\|\rho\|_{W^{m-2,2}}^{m}+(\|\rho\|_{W^{m-2,2}}^{m-1}+1)\|\rho^{(m-2)}\|_{L^2}^{\frac{7}{8}}\|\rho^{(m)}\|_{L^2}^{\frac{1}{8}}+\|\rho^{(m-2)}\|_{L^2}^{\frac{7}{4}}\|\rho^{(m)}\|_{L^2}^{\frac{1}{4}}\Big],
\end{aligned}
\end{equation}
where we used \eqref{st1} and \eqref{st3}.

Notice that \eqref{uxxc} gives $\|\rho\|_{L^\infty(0,T_1; L^2(I))}\leq C(\eta,L)$. By interpolating,
\eqref{st4} and \eqref{st5} lead to
\begin{equation}\label{TT3}
\begin{aligned}
M_1\leq & C(\eta,m)\Big[\|\rho^{(m)}\|_{L^2}^{\frac{5}{8}}\|\rho\|_{\dot{W}^{m-2,2}}^{m}+\|\rho^{(m)}\|_{L^2}^{\frac{1}{8}}\|\rho\|_{\dot{W}^{m-2,2}}^{m}\\
&+\|\rho^{(m)}\|_{L^2}^{\frac{1}{4}}\|\rho\|_{\dot{W}^{m-2,2}}^{m+1}+\|\rho\|_{\dot{W}^{m-2,2}}^{m}+C(\eta,L)\Big]\|\rho^{(m)}\|_{L^2}\\
\leq & \frac{1}{8} \|\rho^{(m)}\|_{L^2}^2+C(\eta,m)\|\rho\|_{\dot{W}^{m-2,2}}^{10m}+C(\eta,L).
\end{aligned}
\end{equation}

Combining \eqref{TT1}, \eqref{TT2}, \eqref{TT3} and Gr\"onwall's inequality, we finally obtain
$$\|u\|_{L^{\infty}([0,T_1];W^{m,2}_{\per_0}(I))}\leq C(\eta,L,\|u^0\|_{W_{\per_0}^{m,2}},T_1),$$
$$\|u\|_{L^2([0,T_1]; W^{m+2,2}_{\per_0}(I))}\leq C(\eta,L,\|u^0\|_{W_{\per_0}^{m,2}},T_1).$$

Step 2. Define $F_\delta : W^{m+2,2}_{\per_0}\rightarrow W^{m+2,2}_{\per_0}$ with
$$F_\delta(u^\delta):=(J_{\delta}*(-\lpi H(\bar{u}_{xx}^\delta)))_x-(J_{\delta}*\Phibp(\bar{u}_{xx}^\delta))_{xx}.$$
We can easily check that $F_\delta$ is locally Lipschitz continuous in $W^{m+2,2}(I)$ for $m\geq1$. Hence by \cite[Theorem 3.1]{Majda2002}, we know \eqref{reg1}
has a unique local solution $u^\delta\in C^1([0,T_0]; W^{m+2,2}_{\per_0}(I))$ and those estimates in Step 1 hold true uniformly in $\delta$. That is, for $T_0$,
we have
\begin{equation}\label{temp1026_1}
\|u^\delta\|_{L^{\infty}([0,T_0];W^{m,2}_{\per_0}(I))}\leq C(\eta,L,\|u^0\|_{W_{\per_0}^{m,2}},T_0),
\end{equation}
\begin{equation}\label{temp1026_2}
\|u^\delta\|_{L^2([0,T_0]; W^{m+2,2}_{\per_0}(I))}\leq C(\eta,L,\|u^0\|_{W_{\per_0}^{m,2}},T_0).
\end{equation}

Since $$E_u^\delta(u^\delta(T))+\int_0^T \int_0^L  u^{\delta 2}_t \ud x \ud t =E_u^\delta(u^\delta(0)),$$
we also have
\begin{equation}\label{temp1026_3}
\|u^\delta_t\|_{L^2([0,T_0]\times I)} \leq C(\eta,L,\|u^0\|_{W_{\per_0}^{m,2}}).
\end{equation}
{  Notice $W^{m+2,2}\hookrightarrow W^{m+1,2}$ compactly and $W^{m+1,2}\hookrightarrow L^2$.
Therefore, as $\delta\rightarrow 0,$  we can use Lions-Aubin's compactness lemma to obtain there exists a subsequence, still denoted as $u^\delta$, such that
$$u^\delta\rightarrow u, \text{ in }L^2([0,T_0];W^{m+1,2}_{\per_0}(I)).$$
And \eqref{temp1026_1}, \eqref{temp1026_2} and \eqref{temp1026_3} show that
$$u\in L^{\infty}([0,T_0];W^{m,2}_{\per_0}(I))\cap L^2([0,T_0];W^{m+2,2}_{\per_0}(I)), $$
$$u_t\in L^{\infty}([0,T_0];W^{m-4,2}_{\per_0}(I)).$$
 Thus we can take limit in \eqref{reg1} and $u$ satisfies \eqref{ua} almost everywhere, i.e., $u$ is the local strong solution of \eqref{ua}. }

Since
$$\|u_t\|_{L^2([0,T_0]\times I)}\leq \liminf_{\delta\rightarrow 0} \|u^\delta_t\|_{L^2([0,T_0]\times I)} \leq C(\eta,L,\|u^0\|_{W_{\per_0}^{m,2}}),$$
$$u_t\in L^2([0,T_0]\times I),$$
by \cite[Theorem 4, p.~288]{Evans1998}, we actually have
$$u\in C([0,T_0];W^{1,2}_{\per_0}(I)).$$

Step 3. We justify the a-priori assumption \eqref{a-p1}. Note that
\begin{equation}\label{uuu}
u_{xx}(x,t)=u_{xx}(0)+\int_0^t u_{xxt}(x,\tau)d\tau,
\end{equation}
and $u_{xx}^0+b \geq \eta$,
so Step 2 and Sobolev embedding theorem lead to
$$u_{xxt}\in L^{\infty}([0,T_0], W^{m-6,2}(I))\hookrightarrow L^{\infty}([0,T_0],L^{\infty}(I)),$$
for $m\geq 7$. Then
$$|\int_0^t u_{xxt}(x,\tau)d\tau|\leq t \|u_{xxt}\|_{L^{\infty}([0,T_0],L^{\infty}(I))}\leq \frac{\eta}{2},\quad t\in[0,T_m],$$
where $T_m<T_0$ depends only on $\eta,\,L$ and $\|u^0\|_{W^{m,2}(I)}$.
This, together with \eqref{uuu}, gives \eqref{a-p1}.
\end{proof}

By using the above Theorem \ref{local}, we now  prove the Theorem \ref{local_h}.
\begin{proof}[Proof of Theorem \ref{local_h}]
Step 1 (Existence).
Assume $h^0\in W^{m,2}_{\per^\star}(I),$ $h_{x}^0\leq \beta$, for some constant $\beta<0$, $m\in \mathbb{Z},\, m\geq 6$. From \eqref{620_2}, there exists $u^0\in W^{m+1,2}_{\per^\star}(I)$ satisfying $u_{xx}^0+b\geq -\beta$. Then by Theorem \ref{local}, there exists $T_m>0$, such that there exists a unique $u$ satisfying \eqref{ua} with the following regularity:
$$u\in L^{\infty}([0,T_m];W^{m+1,2}_{\per_0}(I))\cap L^2([0,T_m];W^{m+3,2}_{\per_0}(I))\cap C([0,T_m];W^{m-3,2}_{\per_0}(I)),$$
$$u_t\in L^{\infty}([0,T_m];W^{m-3,2}_{\per_0}(I)),$$
and $u$ satisfies
$$u_{xx}+b\geq-\frac{\beta}{2},\quad \ale  t\in[0,T_m],\,x\in [0,L].$$
Let $h:=-u_x-bx$. Hence we can get the existence of solution to \eqref{5} satisfying \eqref{620_3} and the regularity stated in Theorem \ref{local_h}.

Step 2 (Uniqueness). Now we assume $h_1,\,h_2$ are two solutions of \eqref{5} satisfying \eqref{620_3} and the same regularity stated in Theorem \ref{local_h}.
Subtract $h_2$-equation from $h_1$-equation and multiply $h_1-h_2$ on both sides. Then integration by parts shows that
\begin{equation}\label{tem20_1}
\begin{aligned}
&\frac{\ud}{\ud t}\int_0^L (h_1-h_2)^2 \ud x=\int_0^L (h_{1t}-h_{2t})(h_1-h_2)\ud x\\
=&\int_0^L -\frac{2\pi}{L}H(h_{1x}-h_{2x})(h_{1xx}-h_{2xx})+\Big[\Big(3h_{1x}+\frac{1}{h_{1x}}\Big)h_{1xx}-\Big(3h_{2x}+\frac{1}{h_{2x}}\Big)h_{2xx}\Big](h_{1xx}-h_{2xx})\ud x\\
=&\int_0^L -\frac{2\pi}{L}H(h_{1x}-h_{2x})(h_{1xx}-h_{2xx})+\Big(3h_{2x}+\frac{1}{h_{2x}}\Big)(h_{1xx}-h_{2xx})^2\\
&+\Big(3h_{1x}+\frac{1}{h_{1x}}-3h_{2x}-\frac{1}{h_{2x}}\Big)h_{1xx}(h_{1xx}-h_{2xx})\ud x\\
=&I_1+I_2+I_3.
\end{aligned}
\end{equation}
Since
\begin{equation}\label{t20_1}
3h_{2x}+\frac{1}{h_{2x}}\leq -2\sqrt{3}, \text{ due to }h_{2x}<0,
\end{equation}
the second term on the right hand of \eqref{tem20_1} is strictly negative, which will be used to control the other two terms.
For $I_1$, notice the property of Hilbert transform $\|H(u)\|_{L^p}\leq c\|u\|_{L^p}$ for $1<p<\infty;$ see \cite[Proposition 9.1.3]{Butzer1971}.
We can use Young's inequality and interpolating to obtain
\begin{equation}\label{t20_2}
I_1\leq \int_0^L \frac{1}{4}(h_{1xx}-h_{2xx})^2 +c(h_1-h_2)^2\ud x.
\end{equation}
To estimate $I_3$, first notice that $h_{1xx}$ is bounded by $\|h_{1}(0)\|_{W^{m,2}}$ and that
\begin{equation*}
\vert h_{1x}\vert\geq -\frac{\beta}{2}>0,\quad \vert h_{2x}\vert\geq -\frac{\beta}{2}>0,
\end{equation*}
due to \eqref{620_3}. Hence
$$\int_0^L\Big[\Big(3h_{1x}-3h_{2x}+\frac{1}{h_{1x}}-\frac{1}{h_{2x}}\Big)h_{1xx}\Big]^2\ud x\leq C(\beta,\|h_{1}(0)\|_{W^{m,2}})(h_{1x}-h_{2x})^2\ud x,$$
where $C(\beta,\|h_{1}(0)\|_{W^{m,2}})$ depends only on $\beta,\,\|h_{1}(0)\|_{W^{m,2}}.$
Then Young's inequality and interpolating show that
\begin{equation}\label{t20_3}
I_3\leq \int_0^L C(\beta,\|h_{1}(0)\|_{W^{m,2}})(h_{1}-h_{2})^2+\frac{1}{4}(h_{1xx}-h_{2xx})^2 \ud x,
\end{equation}
where $C(\beta,\|h_{1}(0)\|_{W^{m,2}})$ depends only on $\beta,\,\|h_{1}(0)\|_{W^{m,2}}.$
Now combining \eqref{t20_1}, \eqref{t20_2}, \eqref{t20_3} with \eqref{tem20_1} leads to
\begin{equation*}
\frac{\ud}{\ud t}\int_0^L (h_1-h_2)^2 \ud x\leq C(\beta,\|h_{1}(0)\|_{W^{m,2}})\int_0^L (h_1-h_2)^2 \ud x.
\end{equation*}
Then by Gr\"onwall's inequality, we have
\begin{equation}
\int_0^L (h_1-h_2)^2 \ud x\leq C(\beta,\|h_{1}(0)\|_{W^{m,2}},T_m)\int_0^L (h_1(0)-h_2(0))^2 \ud x,
\end{equation}
where $C(\beta,\|h_{1}(0)\|_{W^{m,2}},T_m)$ depends only on $\beta,\,\|h_{1}(0)\|_{W^{m,2}}$ and $T_m.$ This gives the uniqueness of the solution to \eqref{5}.
\end{proof}

\subsection{Stability of linearized $\phi$-PDE}\label{sec6.1}

Now we set up the stability of linearized $\phi$-PDE under assumption
$$h_x(0)\in W^{m,2}_{\per_0}(I),\quad h_x(0)\leq 2 \beta <0,$$
with $m\geq 6$.

Recall Theorem \ref{local_h} and Proposition \ref{h-phi}. There exists $T_m>0,$ such that
\begin{equation}\label{04_temp1}
\phi(\alpha,t)\in L^\infty([0,T_m];W^{6,\infty}_{\per^{\star}}(0,1))
\end{equation}
is the strong solution of \eqref{phi3}
 and there exists constants $m_1,\,m_2>0$ such that
\begin{equation}\label{04_temp2}
\begin{aligned}
\phi_\alpha\leq -m_1<0,\quad |\phi^{(i)}|\leq m_2,\,i=1,\cdots,6.
\end{aligned}
\end{equation}

Recall equation \eqref{phi3}:
$$\phi_t=-\phi_\alpha \mu_{xx}=-\partial_\alpha(\frac{1}{\phi_\alpha}(\frac{\delta E}{\delta \phi})_\alpha),
$$
where
\begin{equation*}
\frac{\delta E}{\delta \phi}=\frac{2\pi}{L^2}\PV\int_0^1 \cot\frac{\pi(\phi(\alpha)-\phi(\beta))}{L}\ud \beta -\frac{\phi_{\alpha\alpha}}{\phi_{\alpha}^2}-3\frac{\phi_{\alpha\alpha}}{\phi_{\alpha}^4}.
\end{equation*}
We want to show that the linearized $\phi$-PDE is stable, which will be used in the construction of high-order consistency solution (Section \ref{sec6.3}).

For $\phi,\,\tilde{\phi}$ satisfying equation \eqref{phi3}, set $\phi+\varepsilon \psi=\tilde{\phi}.$
Denote
\begin{equation}\label{A_H}
A:=-\frac{\phi_{\alpha\alpha}}{\phi_{\alpha}^2}-3\frac{\phi_{\alpha\alpha}}{\phi_{\alpha}^4}+\frac{2\pi}{L^2}\PV\int_0^1 \cot\frac{\pi(\phi(\alpha)-\phi(\beta))}{L}\ud \beta ,
\end{equation}
and
\begin{equation}\label{B_H}
B:=\biggl(-\frac{1}{\phi_\alpha^2}-3\frac{1}{\phi_\alpha^4}\biggr)\psi_{\alpha\alpha}
+\biggl(\frac{2\phi_{\alpha\alpha}}{\phi_\alpha^3}+\frac{12\phi_{\alpha\alpha}}{\phi_\alpha^5}\biggr)\psi_{\alpha}-\frac{2\pi^2}{L^3}\PV\int_0^1 \sec^2 \frac{\pi}{L}(\phi(\alpha)-\phi(\beta)) (\psi(\alpha)-\psi(\beta))\ud \beta.
\end{equation}
So the linearized equation of $\phi$-PDE \eqref{phi3} is
\begin{equation}\label{linear_1}
\psi_t=-\partial_\alpha \biggl( -\frac{\psi_\alpha}{\phi_\alpha^2} \partial_\alpha A+\frac{\partial_\alpha B}{\phi_\alpha} \biggr).
\end{equation}
\begin{prop}\label{prop17_1}
Assume $\psi(0)\in L^2_{\per}([0,1])$ and $m_1,m_2>0$ defined in \eqref{04_temp2}. Let $T_m>0$ be the maximal existence time for strong solution $\phi$ in \eqref{04_temp1}. The linearized equation \eqref{linear_1} is stable in the sense
\begin{equation}\label{621_1}
\|\psi(\cdot,t)\|_{L^2_{\per}([0,1])}\leq C(m_1,m_2,T_m)\|\psi(\cdot,0)\|_{L^2_{\per}([0,1])}, \text{ for }t\in[0,T_m],
\end{equation}
where $C(m_1,m_2,T_m)$ is a constant depending only on $m_1,\,m_2,$ and $T_m.$
\end{prop}

\begin{proof}
Step 1. We perform without the Hilbert transform term $\frac{2\pi}{L^2}\PV\int_0^1 \cot\frac{\pi(\phi(\alpha)-\phi(\beta))}{L}\ud \beta .$
Then $A,\,B$ in \eqref{linear_1} become
 $$A:=-\frac{\phi_{\alpha\alpha}}{\phi_{\alpha}^2}-3\frac{\phi_{\alpha\alpha}}{\phi_{\alpha}^4},$$
and
$$B:=\biggl(-\frac{1}{\phi_\alpha^2}-3\frac{1}{\phi_\alpha^4}\biggr)\psi_{\alpha\alpha}
+\biggl(\frac{2\phi_{\alpha\alpha}}{\phi_\alpha^3}+\frac{12\phi_{\alpha\alpha}}{\phi_\alpha^5}\biggr)\psi_{\alpha}.$$

Because $\psi$ is $1$-periodic function respect to $\alpha$,
we have
\begin{equation*}
\begin{aligned}
\psi_t=&-\partial_\alpha \biggl( -\frac{\psi_\alpha}{\phi_\alpha^2} A_\alpha+\partial_\alpha \biggl(\frac{B}{\phi_\alpha}\biggr)- \biggl(\frac{1}{\phi_\alpha}\biggr)_\alpha B\biggr)\\
=&-\partial_{\alpha\alpha}\biggl(\frac{B}{\phi_\alpha}\biggr)+\partial_\alpha \biggl( \frac{\psi_\alpha}{\phi_\alpha^2} A_\alpha-\frac{\phi_{\alpha\alpha}}{\phi_\alpha^2} B\biggr)\\
=&-\partial_{\alpha\alpha}\bigg[ \biggl(-\frac{1}{\phi_\alpha^3}-\frac{3}{\phi_\alpha^5}\biggr)\psi_{\alpha\alpha}+ \biggl(\frac{2\phi_{\alpha\alpha}}{\phi_\alpha^4}+\frac{12\phi_{\alpha\alpha}}{\phi_\alpha^6}\biggr)\psi_\alpha\bigg]\\
&+\partial_\alpha \bigg[\biggl(\frac{\phi_{\alpha\alpha}}{\phi_\alpha^4}+\frac{3\phi_{\alpha\alpha}}{\phi_\alpha^6}\biggr)\psi_{\alpha\alpha}+\biggl(-\frac{2\phi_{\alpha\alpha}^2}{\phi_\alpha^5}-\frac{12\phi_{\alpha\alpha}^2}{\phi_\alpha^7}+\frac{A_\alpha}{\phi_{\alpha}^2}\biggr)\psi_\alpha\bigg].
\end{aligned}
\end{equation*}
Multiplying both sides by $\psi$ and integration by parts show that
\begin{align}\label{617_1}
\int_0^1 \psi\psi_t \ud\alpha =&\int_0^1 \Bigg[ \big(\frac{1}{\phi_\alpha^3}+\frac{3}{\phi_\alpha^5}\big)\psi_{\alpha\alpha}^2
-\big(\frac{3\phi_{\alpha\alpha}}{\phi_\alpha^4}+\frac{15\phi_{\alpha\alpha}}{\phi_\alpha^6}\big)\psi_\alpha\psi_{\alpha\alpha}
+\bigg( \frac{2\phi_{\alpha\alpha}^2}{\phi_\alpha^5}+\frac{12\phi_{\alpha\alpha}^2}{\phi_\alpha^7}-\frac{A_\alpha}{\phi_\alpha^2}\bigg) \psi_\alpha^2\Bigg]\ud\alpha.
\end{align}

From Young's inequality, for any $\delta,\varepsilon>0,$ we have
\begin{equation}\label{615_4}
\psi_{\alpha\alpha}\psi_\alpha\leq \varepsilon \psi_{\alpha\alpha}^2+\frac{1}{4\varepsilon}\psi_\alpha^2,
\end{equation}
and
\begin{equation}\label{615_5}
\int_0^1 \psi_\alpha^2\ud\alpha \leq  \int_0^1\biggl(\delta\psi_{\alpha\alpha}^2+\frac{1}{4\delta}\psi^2\biggr)\ud\alpha.
\end{equation}

Note that $\phi_\alpha$ is negative and from \eqref{04_temp1},\,\eqref{04_temp2}, we know
$$\frac{1}{\phi_\alpha^3}+\frac{3}{\phi_\alpha^5}\leq -\biggl(\frac{1}{m_2^3}+\frac{1}{m_2^5}\biggr).$$

Now choose $\varepsilon$, $\delta$ in \eqref{615_4} and \eqref{615_5} such that the last two terms in \eqref{617_1} can be controlled by $\int_0^1 -\biggl(\frac{1}{m_2^3}+\frac{1}{m_2^5}\biggr)\psi_{\alpha\alpha}^2 +C(m_1,m_2)\psi^2\ud\alpha.$
Therefore, combining \eqref{615_4}, \eqref{615_5} and \eqref{04_temp2}, we have
\begin{align}\label{psi_S}
\frac{\ud }{\ud t}\int_0^1 \psi^2 \ud\alpha +C(m_2)\int_0^1 \psi_{\alpha\alpha}^2 \ud\alpha\leq \int_0^1 C(m_1,m_2)\psi^2 \ud\alpha,
\end{align}
where $C(m_2),\,C(m_1,m_2)>0$ are constants depending on $m_1,\,m_2$.

By Gr\"onwall's inequality, we finally achieve the stability for $\psi$ in the sense of \eqref{621_1}.

Step 2. If we consider Hilbert transform, then $A,\,B$ are defined in \eqref{A_H} and \eqref{B_H}.
First notice that change of variable from $h$ to $\phi$ does not effect the Cauchy principal value integral and that $h_x<0$. Then for any $\alpha\in[0,1],$ by variable substitution, we have
\begin{equation}\label{623_1}
\begin{aligned}
&\PV \int_0^1 \frac{\pi}{L}\cot \Big( \frac{\pi}{L} \big(\phi(\alpha)-\phi(\beta)\big) \Big)\ud \beta
=-\PV \int_0^1 \sum_{k\in \mathbb{Z}}\frac{1}{\phi(\beta)-\phi(\alpha)-kL}\ud \beta\\
=&-\PV\int_{-\infty}^{+\infty}\frac{1}{\phi(\beta)-\phi(\alpha)}\ud \beta
=\PV\int_{-\infty}^{+\infty}\frac{h_y}{y-x}\ud y\\
=&\PV\sum_{k\in\mathbb{Z}}\int_{-\frac{L}{2}+kL}^{\frac{L}{2}+kL}\frac{h_y}{y-x}\ud y
=\frac{\pi}{L} \PV \int_{-\frac{L}{2}}^{\frac{L}{2}}h_y \cot(\frac{y-x}{L}\pi)\ud y\\
=&-\pi H(h_x)\circ \phi,
\end{aligned}
\end{equation}
where we used the relation for Hilbert kernel
$$\sum_{k\in\mathbb{Z}}\frac{1}{x+kL}=\frac{\pi}{L}\cot(\frac{\pi}{L}x).$$
Hence
\begin{align*}
\Big(\PV\int_0^1 \cot\frac{\pi}{L}(\phi(\alpha)-\phi(\beta))\ud \beta\Big)_{\alpha}=-L(H(h_{xx})\circ\phi) \phi_\alpha
\end{align*}
is $L^p$ bounded due to the property of Hilbert transform $H(u)_x=H(u_x)$ for $u_x\in L^p$ with $1<p<\infty.$

Second, using the periodicity of $\psi$, integration by parts shows that
\begin{align*}
&\frac{\pi}{L}\PV\int_0^1 \sec^2 \frac{\pi}{L}(\phi(\alpha)-\phi(\beta)) (\psi(\alpha)-\psi(\beta))\ud \beta\\
=&\PV \int_0^1 \cot\frac{\pi(\phi(\alpha)-\phi(\beta))}{L}\bigg[-\frac{\psi_\alpha(\beta)}{\phi_\alpha^2(\beta)}-\frac{(\psi(\alpha)-\psi(\beta))\phi_{\alpha\alpha}(\beta)}{\phi_\alpha^2(\beta)}\bigg] \ud\beta.
\end{align*}
For any $\varepsilon>0$, by Young's inequality, we have
\begin{align*}
&\int_0^1 \PV\int_0^1 \psi_{\alpha\alpha}(\alpha)(\psi(\alpha)-\psi(\beta))\sec^2 \frac{\pi}{L}(\phi(\alpha)-\phi(\beta)) \ud \beta  \ud\alpha\\
\leq & 2\varepsilon \int_0^1 \psi_{\alpha\alpha}^2 \ud\alpha +\frac{c}{\varepsilon}\int_0^1 \bigg[\PV \int_0^1 \cot\frac{\pi(\phi(\alpha)-\phi(\beta))}{L}\bigg(-\frac{\psi_\alpha(\beta)}{\phi_\alpha^2(\beta)}-\frac{(\psi(\alpha)-\psi(\beta))\phi_{\alpha\alpha}(\beta)}{\phi_\alpha^2(\beta)}\bigg) \ud\beta\bigg]^2  \ud\alpha.
\end{align*}
Similar to \eqref{623_1}, we have
\begin{align*}
&\PV \int_0^1 \cot\frac{\pi(\phi(\alpha)-\phi(\beta))}{L}\bigg(-\frac{\psi_\alpha(\beta)}{\phi_\alpha^2(\beta)}-\frac{(\psi(\alpha)-\psi(\beta))\phi_{\alpha\alpha}(\beta)}{\phi_\alpha^2(\beta)}\bigg) \ud\beta\\
=&\bigg[H(-\frac{\psi_\alpha}{\phi_\alpha^3}\circ h)+H(\frac{\phi_{\alpha\alpha}\psi}{\phi_\alpha^3}\circ h)+\psi(\alpha)H(\frac{-\phi_{\alpha\alpha}}{\phi_{\alpha}^3}\circ h)\bigg]\circ\phi.
\end{align*}
Then notice the property of Hilbert transform $\|H(u)\|_{L^p}\leq c\|u\|_{L^p}$ for $1<p<\infty;$ see \cite[Proposition 9.1.3]{Butzer1971}.
For any $\varepsilon,\delta>0$,
by H\"older's inequality and interpolating, we have
\begin{align*}
&\int_0^1 \PV\int_0^1 \psi_{\alpha\alpha}(\alpha)(\psi(\alpha)-\psi(\beta))\sec^2 \frac{\pi}{L}(\phi(\alpha)-\phi(\beta)) \ud \beta  \ud\alpha\\
\leq & 2\varepsilon \int_0^1 \psi_{\alpha\alpha}^2 \ud\alpha +\frac{c}{\varepsilon}\int_0^1\bigg[H(-\frac{\psi_\alpha}{\phi_\alpha^3}\circ h)+H(\frac{\phi_{\alpha\alpha}\psi}{\phi_\alpha^3}\circ h)+\psi(\alpha)H(\frac{-\phi_{\alpha\alpha}}{\phi_{\alpha}^3}\circ h) \bigg]^2\circ \phi\ud\alpha\\
\leq & 2\varepsilon \int_0^1 \psi_{\alpha\alpha}^2 \ud\alpha +\frac{c}{\varepsilon}\int_0^1\bigg[\frac{\psi_\alpha^2}{\phi_\alpha^5}+\frac{\phi_{\alpha\alpha}^2\psi^2}{\phi_\alpha^5}\bigg]\ud \alpha +\big(\int_0^1\psi^4(\alpha)\ud \alpha\big)^{\frac{1}{2}}\big(\int_0^1\frac{\phi_{\alpha\alpha}^4}{\phi_{\alpha}^{11}}\ud \alpha\big)^{\frac{1}{2}}\\
\leq &  \Big(2\varepsilon+\frac{\delta}{\varepsilon}\Big) \int_0^1 \psi_{\alpha\alpha}^2 \ud\alpha+\frac{C(m_1,m_2)}{\varepsilon\delta}\int_0^1 \psi(\alpha)^2 \ud \alpha
\end{align*}
where $C(m_1,m_2)$ depends only on $m_1,\,m_2.$ Here we used variable substitution twice and \eqref{04_temp2}.

Then we can perform just like Step 1 to get \eqref{psi_S} and complete the proof of Proposition \ref{prop17_1}.
\end{proof}

\section{Modified BCF type model}\label{sec5}

We want to rigorously study the continuum limit of a BCF type model and figure out the convergence rate.
From now on, we assume the initial data $x_i(0)$ satisfying
\begin{equation}\label{initial}
  x_i(0)<x_{i+1}(0), \text{ for }i=1,\cdots,N.
\end{equation}

As mentioned in the Introduction,
we need to modify the ODE as follows
\begin{equation}\label{ODE}
\frac{\ud x_i}{\ud t}=\frac{1}{a} \biggl(\frac{f_{i+1}-f_{i}}{x_{i+1}-x_i}-\frac{f_{i}-f_{i-1}}{x_{i}-x_{i-1}}\biggr),\quad i=1,\cdots,N,
\end{equation}
where the chemical potential
\begin{equation}\label{G}
\begin{aligned}
f_i:=-\frac{2}{L}\sum_{j\neq i}\frac{a}{x_j-x_i}+\biggl(\frac{1}{x_{i+1}-x_i}-\frac{1}{x_{i}-x_{i-1}}\biggr) +\biggl(\frac{a^2}{(x_{i+1}-x_i)^3}-\frac{a^2}{(x_i-x_{i-1})^3}\biggr),
\end{aligned}
\end{equation}
for $i=1,\cdots,N.$ Notice \eqref{ODE} with \eqref{G} is exact the ODE \eqref{ODE22} with \eqref{622_1}, so we refer \eqref{ODE} in the following.

From now on, keep in mind the relation between the Hilbert kernel and Cauchy kernel is
\begin{equation}\label{09temp1}
\sum_{k\in\mathbb{Z}}\frac{1}{x+kL}=\frac{\pi}{L}\cot(\frac{\pi}{L}x).
\end{equation}
The corresponding energy is
\begin{equation}\label{Energy}
E^N:=a^2\sum_{1\leq i<j\leq N} \frac{2}{L}\ln\big\lvert\sin  (\frac{\pi}{L}(x_j-x_i))|+a\sum_{i=0}^N\biggl(-\ln|\frac{x_i-x_{i+1}}{a}|+\frac{a^2}{2}\frac{1}{(x_i-x_{i+1})^2}\biggr).
\end{equation}
Since as $a \to 0$, we have $x_i = O(a)$, so the contribution of the various terms in $E^N$ is on the same order.
We have
$$ f_i=\frac{1}{a}\frac{\partial E^N}{\partial x_i},$$
and energy identity
\begin{equation}\label{EN_I}
\frac{\ud E^N}{\ud t}+\sum_{i=1}^N \frac{(f_{i+1}-f_i)^2}{x_{i+1}-x_i}=0,
\end{equation}
which is analogous to \eqref{E3.8}.

%
%
%

We will first study some properties of \eqref{ODE} and obtain the consistence result in Section \ref{consis}. Then we construct an auxiliary solution with high-order consistency in Section \ref{sec6.3}, which is important when we prove the convergence rate of the modified ODE system. After those preparations, the proof of Theorem \ref{thm_main} will be given in Section \ref{sec6.4}.

\subsection{Global solution of ODE}

In this section, we will prove that for any fixed $N\geq 2$, the ODE system \eqref{ODE} has
a global in time solution.
\begin{prop}\label{prop625}
Assume initial data satisfy \eqref{initial}. Then for any $N\geq 2$, the ODE system \eqref{ODE} has
a global in time solution.
\end{prop}

\begin{proof}
Let $T_{\max}$ be the maximal existence time. Then if $T_{\max}<+\infty$, from standard Extension Theory for ODE, we know either two steps collide, i.e.
there exists $i$, such that $x_i(T_{\max})=x_{i+1}(T_{\max});$ or step reaches infinity, i.e. $x_i(T_{\max})=+\infty.$

Denote
$$\ell_{min}(t):=\min_{i\in\mathbb{N}}\{x_{i+1}(t)-x_i(t)\},$$
and we state a proposition that we have a positive lower bound for
$\ell_{min}(t).$ We will proof this proposition later.

\begin{prop}\label{lem5.1}
For any $N\geq 2$, assume initial data satisfy \eqref{initial} and system \eqref{ODE} has initial energy $E^N(0)$. Then for any time $t$ the solution of \eqref{ODE} exist, we have
$$\ell_{min}(t)\geq C(N)>0,$$
where $C(N)$ is a constant depending only on $N$.
\end{prop}

By Proposition \ref{lem5.1}, we have
$$\ell_{\min}(T_{\max})\geq \lim_{t\rightarrow T_{\max}} \ell_{\min}(t)\geq C(N)>0,$$
which contradicts with $x_i(T_{\max})=x_{i+1}(T_{\max}).$

On the other hand, combining Proposition \ref{lem5.1} with equation \eqref{ODE} gives
$$\max_{1\leq i\leq N}\lvert\dot{x}_i \rvert\leq C(N),$$
where $C(N)$ is a constant depending only on $N$. Hence there will be no finite time blow up and we conclude $T_{\max}=+\infty.$
\end{proof}

\begin{proof}[Proof of Proposition \ref{lem5.1}.]

  First from \eqref{EN_I}, we know, for any time $t$ the solution exist,
$$E^N(t)\leq E^N(0).$$

Let $0<\ell^\star\leq 1$ small enough. Then
$$\frac{2\pi}{L^2}\cot\frac{\pi}{L}\ell-\frac{1}{2}\frac{a^2}{\ell^3}<0,\quad \text{for }0<\ell\leq \ell^\star.$$
Thus, at least for $0<\ell\leq \min\{\ell^\star,\frac{L}{2}\}$, we know
$$g(\ell):=\frac{2}{L}\ln \Big(\sin \frac{\pi}{L}\ell\Big)+\frac{a^2}{4\ell^2}$$
is positive, i.e.
$$\frac{2}{L}\ln \sin \frac{\pi}{L}\ell+\frac{a^2}{4\ell^2}> 0.$$

Hence
$$\frac{2}{L}\ln \sin \frac{\pi}{L}\ell+\frac{a^2}{2\ell^2}> \frac{a^2}{4\ell^2} ,$$
and
$$\frac{2}{L}\ln(\sin\frac{\pi}{L}\ell)-\ln\big(\frac{\ell}{a}\big)+\frac{a^2}{2\ell^2}> \frac{a^2}{4\ell^2}+\ln a\geq c_0(N),$$
where $c_0(N)$ is a constant depending only on $N$.
Then
we obtain
\begin{align*}
E^N\geq& a^2\bigg[\frac{2}{L}\ln(\sin(\frac{\pi}{L}\ell_{min})) -\ln(\ell_{min})+\ln a+\frac{a^2}{2\ell^2_{min}}+(\frac{N(N-1)}{2}-1)c_0(N)\bigg]\\
\geq & \frac{a^4}{4\ell_{min}^2}+c_1(N),
\end{align*}
where $c_1(N)$ is a constant depending only on $N$.

Therefore, we have
$$\frac{1}{\ell_{min}^2}\leq C(N,E^N(0)),$$
where $C(N,E^N(0))$ is a positive constant depending only on $N$ and initial data.

So
we finally get
$$\ell_{min}\geq \min\{\frac{L}{2},\ell^\star, \frac{1}{\sqrt{C(N,E^N(0))}} \}.$$
\end{proof}

\section{Consistency}\label{consis}

In this section, we study the local consistency between exact solution $\phi$ of equation \eqref{phi3} and solution $x$ of equation \eqref{ODE}.
From now on, we always assume there exists a constant $\beta<0$ such that the initial data satisfy
$$h_x(0)\in W^{m,2}_{\per_0}(I),\quad h_x(0)\leq 2 \beta <0,$$
with $m\geq 6$.

From Theorem \ref{local_h}, we know there exists $T_m>0$, for $t\in[0,T_m],$ 
$h(x,t)\in L^\infty([0,T];W^{6,\infty}_{\per^{\star}}(\mathbb{R})) $ is the strong solution of \eqref{5} and
\begin{equation}\label{eta}
\begin{aligned}
h_x\leq \beta<0.
\end{aligned}
\end{equation}
Also by Proposition \ref{h-phi}, we know $\phi(\alpha,t)$ is the strong solution of \eqref{phi3} satisfying \eqref{04_temp1} and \eqref{04_temp2}.

Denote
\begin{equation}\label{fbar}
\begin{aligned}
\bar{f}_i:=-\frac{2}{L}\sum_{j\neq i}\frac{a}{\phi_j-\phi_i}+\biggl(\frac{1}{\phi_{i+1}-\phi_i}-\frac{1}{\phi_{i}-\phi_{i-1}}\biggr)
+\biggl(\frac{a^2}{(\phi_{i+1}-\phi_i)^3}-\frac{a^2}{(\phi_i-\phi_{i-1})^3}\biggr).
\end{aligned}
\end{equation}
The main result in this section is Theorem \ref{th4.4}:
\begin{thm}\label{th4.4}
For all $i=1,\cdots,N$, let $\bar{f}_i$ be defined in \eqref{fbar},
and
\begin{equation}\label{ralpha0}
v_1(\alpha;\phi):=-\frac{\phi_{\alpha\alpha}}{2\phi_\alpha^2}(\alpha),\quad
r_0(\alpha;\phi):=\biggl(\frac{v_{1\alpha} \phi_{\alpha\alpha}}{\phi_\alpha^2}-\frac{v_{1\alpha\alpha}}{\phi_\alpha}\biggr)(\alpha).
\end{equation}
Then we have
\begin{equation}\label{consistency1}
\frac{\ud \phi_i}{dt}=\frac{1}{a} \biggl(\frac{\bar{f}_{i+1}-\bar{f}_{i}}{\phi_{i+1}-\phi_i}-\frac{\bar{f}_{i}-\bar{f}_{i-1}}{\phi_{i}-\phi_{i-1}}\biggr)+r_0(\alpha_i;\phi)a+R_i a^2, \quad t\in[0,T],
\end{equation}
and
\begin{equation}\label{consistency2}
|r_0(\alpha_i;\phi)|\leq C( \beta,\|h(0)\|_{W^{7,2}(I)}),\quad |R_i|\leq C( \beta,\|h(0)\|_{W^{7,2}(I)}),
\end{equation}
where $C( \beta,\|h(0)\|_{W^{7,2}(I)})$ depends on $\beta,\|h(0)\|_{W^{7,2}(I)}$, and $R_i$ is defined in \eqref{0502t1}.
In addition, we have
$$\frac{\ud E^N(\phi)}{\ud t}+a \sum_{i=1}^N \biggl(\frac{\bar{f}_{i+1}-\bar{f}_i}{a}\biggr)^2\leq Ca.$$
\end{thm}

To achieve this goal, first we need to set up some notations and lemmas.

From \eqref{04_temp1} and \eqref{04_temp2}, there exist constants $c_1,\,c_2>0$, such that
\begin{equation}\label{y}
c_1 a \leq \phi_{i+1}-\phi_i \leq c_2 a.
\end{equation}

Denoting
\begin{equation}\label{Fi}
F_i:=\frac{1}{a} \biggl(\frac{\bar{f}_{i+1}-\bar{f}_{i}}{\phi_{i+1}-\phi_i}-\frac{\bar{f}_{i}-\bar{f}_{i-1}}{\phi_{i}-\phi_{i-1}}\biggr),
\end{equation}
we want to estimate the difference between $F_i$ and $\frac{\ud
  \phi_i}{\ud t}$. From PDE \eqref{5} and \eqref{fan}, we
have
\begin{equation}\label{4.2}
\frac{\ud \phi_i}{\ud t}=-\frac{(-\lpi H(h_x)+(\frac{1}{h_x}+3h_x)h_{xx})_{xx}}{h_x}\biggr|_{\phi_i}.
\end{equation}
The main task is then to calculate the term $F_i$. Let us first
estimate $\bar{f}_i$ till order $a$ accuracy by writing
\begin{equation*}
  \bar{f}_i = I_{1, i} + I_{2, i} + I_{3, i},
\end{equation*}
where
\begin{equation}\label{I123}
\begin{aligned}
I_{1,i}&:=-\frac{2}{L}\sum_{j\neq i}\frac{a}{\phi_j-\phi_i}=-\frac{2}{L}\sum_{k\in\mathbb{Z}}\sum_{j= 1\atop j\neq i}^{N}\frac{a}{\phi_j-\phi_i+kL},\\
I_{2,i}&:=\frac{1}{\phi_{i+1}-\phi_i}-\frac{1}{\phi_{i}-\phi_{i-1}},\\
I_{3,i}&:=\frac{a^2}{(\phi_{i+1}-\phi_i)^3}-\frac{a^2}{(\phi_i-\phi_{i-1})^3}.
\end{aligned}
\end{equation}
To simplify notations, we will henceforth denote
$$\varphi_i=\varphi(x)|_{x=x_i}.$$
Next, we state four lemmas to estimate $I_{1,i},\,I_{2,i},\,I_{3,i}$ one by one, from which, we know $O(a)$ error only show up when estimating the first term $I_{1,i}$ in Lemma \ref{ll}.

\begin{lem}\label{l4.1}
Let $I_{2,i}$ be defined in \eqref{I123} and
$v_2$ be function of $\alpha$ defined as
\begin{equation}\label{u2}
v_2(\alpha;\phi):=-\frac{\phi^{(4)}}{12\phi_\alpha^2}+\frac{\phi_{\alpha\alpha}}{\phi_\alpha^4}\Bigl(\frac{1}{3}\phi_\alpha \phi^{(3)}-\frac{1}{4}\phi_{\alpha\alpha}^2\Bigr).
\end{equation}
Then we have
\begin{equation}\label{I2}
I_{2,i}=\frac{h_{xx}}{h_x}\biggr|_{\phi_i}+v_2(\alpha_i;\phi) a^2+R_{2,i},
\end{equation}
 where $|R_{2,i}|\leq  a^4 C( \beta,\|h(0)\|_{W^{7,2}(I)})$.
\end{lem}

\begin{proof}
Notice we have
\begin{align}
& \phi_{i+1}=\phi_i-\phi_{\alpha,i} a+\frac{1}{2} \phi_{\alpha\alpha,i} a^2-\frac{1}{3!}\phi^{(3)}_i a^3 + \frac{1}{4!}\phi^{(4)}_i a^4- \frac{1}{5!}\phi^{(5)}_i a^5+\frac{1}{6!}\phi^{(6)}(\xi^+)a^6,\\
& \phi_{i-1}=\phi_i+\phi_{\alpha,i} a+\frac{1}{2} \phi_{\alpha\alpha,i} a^2+\frac{1}{3!}\phi^{(3)}_i a^3 + \frac{1}{4!}\phi^{(4)}_i a^4+ \frac{1}{5!}\phi^{(5)}_i a^5+\frac{1}{6!}\phi^{(6)}(\xi^-)a^6,
\end{align}
where $\xi^+\in[\alpha_i,\alpha_{i+1}],\,\xi^-\in[\alpha_{i-1},\alpha_i].$

Hence, using \eqref{fan}, we have
\begin{equation*}
\begin{aligned}
I_{2,i}&= \frac{1}{\phi_{i+1}-\phi_i}-\frac{1}{\phi_i-\phi_{i-1}} \\
&= \frac{\frac{2\phi_i-\phi_{i+1}-\phi_{i-1}}{a^2}}{(\frac{\phi_{i+1}-\phi_i}{a})(\frac{\phi_i-\phi_{i-1}}{a})}\\
&= \biggl(-\phi_{\alpha\alpha,i}-\frac{1}{12}\phi^{(4)}_i a^2-\frac{1}{6!}(\phi^{(6)}(\xi^+)+\phi^{(6)}(\xi^-)) a^4\biggr)\\
&\quad\cdot\frac{1}{-\phi_{\alpha,i} +\frac{1}{2} \phi_{\alpha\alpha,i} a-\frac{1}{3!}\phi^{(3)}_i a^2 + \frac{1}{4!}\phi^{(4)}_i a^3- \frac{1}{5!}\phi^{(5)} a^4 + \frac{1}{6!} \phi^{(6)}(\xi^+) a^5}\\
&\quad\cdot\frac{1}{-\phi_{\alpha,i} -\frac{1}{2} \phi_{\alpha\alpha,i} a-\frac{1}{3!}\phi^{(3)}_i a^2 - \frac{1}{4!}\phi^{(4)}_i a^3- \frac{1}{5!}\phi^{(5)} a^4 -   \frac{1}{6!} \phi^{(6)}(\xi^-) a^5 }\\
&= \frac{-\phi_{\alpha\alpha,i}-\frac{1}{12}\phi^{(4)}_i a^2-\frac{1}{6!}(\phi^{(6)}(\xi^+)+\phi^{(6)}(\xi^-)) a^4}{(\phi_{\alpha,i}^{2} +A_1(\alpha_i;\phi) a^2 + A_{2,i} a^4)}\\
&= \biggl(\frac{-\phi_{\alpha\alpha}}{\phi_\alpha^2}\biggr)_i+\biggl(-\frac{\phi^{(4)}}{12\phi_\alpha^2}+\frac{\phi_{\alpha\alpha}}{\phi_\alpha^4}A_1\biggr)_i a^2+A_{3,i} a^4\\
&= \biggl(\frac{h_{xx}}{h_x^2}\biggr)_i+\biggl(-\frac{\phi^{(4)}}{12\phi_\alpha^2}+\frac{\phi_{\alpha\alpha}}{\phi_\alpha^4}A_1\biggr)_i a^2+A_{3,i} a^4,
\end{aligned}
\end{equation*}
where
$$A_1(\alpha;\phi)=\frac{1}{3}\phi_\alpha \phi^{(3)}-\frac{1}{4}\phi_{\alpha\alpha}^2,\quad |A_{2,i}|\leq c,\quad |A_{3,i}|\leq c.$$
Denote
\begin{equation}
v_2(\alpha;\phi)=-\frac{\phi^{(4)}}{12\phi_\alpha^2}+\frac{\phi_{\alpha\alpha}}{\phi_\alpha^4}A_1,
\end{equation}
we complete the proof of Lemma \ref{l4.1}.
\end{proof}

Now we claim an approximation for periodic Hilbert transform.
\begin{lem}\label{l4.2}
For any $\phi(\alpha_i)$, $i=1,\cdots,N,$  we have
\begin{equation}\label{lq4.2}
\PV\int_0^1 \frac{\pi}{L}\cot(\frac{\pi}{L}(\phi(\alpha_i)-\phi(\alpha)))\ud\alpha
=\sum_{j\neq i,j=1}^{N}a \frac{\pi}{L}\cot(\frac{\pi}{L}(\phi(\alpha_i)-\phi_j))+\frac{a}{2}\frac{\phi_{\alpha\alpha}}{\phi_\alpha^2}+R_{1,i},
\end{equation}
where $|R_{1,i}|\leq  a^4C( \beta,\|h(0)\|_{W^{7,2}(I)})$.
\end{lem}

\begin{proof} We use the Euler-Maclaurin expansion in \cite{Sidi1988} to estimate $R_{1,i}$. Without loss of generality, we assume $i=1,\cdots, N-1,$ that is $\alpha_i\neq 0,\,1$. For $i=N,$ we can change interval $[0,1]$ to $[-a,1-a]$ due to periodicity. Using \eqref{09temp1}, we can see
 \begin{equation*}
\begin{aligned}
&\PV\int_{0}^{1}\frac{\pi}{L}\cot(\frac{\pi}{L}(\phi(\alpha)-\phi(\alpha_i)))\ud\alpha\\
=&\sum_{k\in \mathbb{Z}}\PV\int_0^1 \frac{1}{\phi(\alpha)-\phi(\alpha_i)+kL}\ud\alpha\\
=&\PV\int_0^1 \frac{1}{\phi(\alpha)-\phi(\alpha_i)}\ud\alpha+\sum_{k\in\mathbb{Z} \atop k\neq 0}\int_0^1 \frac{1}{\phi(\alpha)-\phi(\alpha_i)+kL}\ud\alpha\\
=&T_1+T_2.
\end{aligned}
\end{equation*}

Denote
$$^{\#}\sum_{j=0}^{N}\beta_j=\sum_{j=1}^{N-1}\beta_j+\frac{1}{2}\sum_{j=0,N}\beta_j.$$

First we recall Theorem 1 and Theorem 4 in \cite{Sidi1988} as follows:
\begin{thm}[Theorem~1 of \cite{Sidi1988}]\label{T1}
Let function $g(x)$ be $2m$ times differentiable on $[0,1]$. Then
  $$\int_0^1g(x)\ud x=a^{\#}\sum_{j=0}^{N} g(x_j)+\sum_{\mu=1}^{m-1}\frac{B_{2\mu}}{2\mu!}[g^{(2\mu-1)}|_{x=1}^{x=0}]a^{2\mu}+R_{2m}[g;(0,1)],$$
  where $$R_{2m}[g;(0,1)]=a^{2m}\int_0^1 \frac{\bar{B}_{2m}[\frac{x}{a}]-B_{2m}}{(2m)!}g^{(2m)}(x)\ud x,$$
  $B_{\mu}$ is the Bernoulli number and $\bar{B}_{\mu}$ is the periodic Bernoullian function of order $\mu$.
\end{thm}

\begin{thm}[Theorem~4 of \cite{Sidi1988}]\label{T2}
  Let function $G(x)$ be $2m$ times differentiable on $[0,1]$ and let $g(x)=\frac{G(x)}{x-t}$. Then
  $$\int_0^1g(x)\ud x=a^{\#}\sum_{j=0,x_j\neq t}^{N} g(x_j)+aG'(t)+\sum_{\mu=1}^{m-1}\frac{B_{2\mu}}{2\mu!}[g^{(2\mu-1)}|_{x=1}^{x=0}]a^{2\mu}+\tilde{R}_{2m}[g;(0,1)],$$
  where $$\tilde{R}_{2m}[g;(0,1)]=a^{2m}\PV\int_0^1 \frac{\bar{B}_{2m}[\frac{x}{a}]-B_{2m}}{(2m)!}g^{(2m)}(x)\ud x.$$
\end{thm}

For the nonsingular $T_2$, we apply Theorem \ref{T1} to obtain
\begin{equation}\label{617_2}
T_2=\sum_{k\in\mathbb{Z} \atop k\neq 0}\biggl[ a(^{\#}\sum_{j=0}^{N}\frac{1}{\phi(\alpha_j)-\phi(\alpha_i)+kL})+ a^2 \frac{B_2}{2}\frac{\ud }{\ud\alpha}\biggl(\frac{1}{\phi(\alpha)-\phi(\alpha_i)+kL}\biggr)\biggr|_{\alpha=1}^{\alpha=0}+a^4 e_1(k) \biggr],
\end{equation}
where
\begin{equation}\label{617_3}
|e_1(k)|=\bigg\lvert\int_0^1 \frac{\bar{B}_4[\frac{\alpha}{a}]-B_4}{4!}\frac{\ud ^4}{\ud  \alpha^4}\biggl(\frac{1}{\phi(\alpha)-\phi(\alpha_i)+kL}\biggr) \ud \alpha\bigg\rvert \leq c \max_{\alpha\in [0,1]}\frac{\ud ^4}{\ud  \alpha^4}\biggl(\frac{1}{\phi(\alpha)-\phi(\alpha_i)+kL}\biggr).
\end{equation}
Due to $\phi_\alpha(1)-\phi_\alpha(0)=0$, the second term in \eqref{617_2} becomes
\begin{equation}\label{K2}
\begin{aligned}
K_{2}:=&\sum_{k\in\mathbb{Z} \atop k\neq 0}\frac{B_2}{2}\frac{\ud }{\ud\alpha}\biggl(\frac{1}{\phi(\alpha)-\phi(\alpha_i)+kL}\biggr)\biggr|_{\alpha=1}^{\alpha=0}\\
=&\sum_{k\in\mathbb{Z} \atop k\neq 0}\frac{B_2}{2}\phi_\alpha(0)\biggl(\frac{1}{(kL-\phi(\alpha_i))^2}-\frac{1}{(L+kL-\phi(\alpha_i))^2}\biggr).
\end{aligned}
\end{equation}

To estimate the last term in \eqref{617_2}, since $\max_{\alpha\in [0,1]}\frac{\ud ^4}{\ud  \alpha^4}\biggl(\frac{1}{\phi(\alpha)-\phi(\alpha_i)+kL}\biggr)$ in \eqref{617_3} is summable respect to $k$, we get
\begin{equation}\label{e1}
|\sum_{k\in\mathbb{Z} \atop k\neq 0} e_1(k)|\leq C( \beta,\|h(0)\|_{W^{7,2}(I)}).
\end{equation}

Now we deal with the singular term $T_1$.
Denote
$G(\alpha):=\frac{\alpha-\alpha_i}{\phi(\alpha)-\phi(\alpha_i)}.$
 Applying Theorem \ref{T2} to
 $$g(\alpha)=\frac{G(\alpha)}{\alpha-\alpha_i}=\frac{\frac{\alpha-\alpha_i}{\phi(\alpha)-\phi(\alpha_i)}}{\alpha-\alpha_i}=\frac{1}{\phi(\alpha)-\phi(\alpha_i)},$$
then we have
\begin{equation}\label{617_4}
\begin{aligned}
T_1= a(^{\#}\sum_{j=0,j\neq i}^{N}\frac{1}{\phi(\alpha_j)-\phi(\alpha_i)})-\frac{a}{2}\frac{\phi_{\alpha\alpha}}{\phi_\alpha^2}\biggr|_{\alpha_i}+ a^2 \frac{B_2}{2}\frac{\ud }{\ud\alpha}\biggl(\frac{1}{\phi(\alpha)-\phi(\alpha_i)}\biggr)\biggr|_{\alpha=1}^{\alpha=0}+a^4 e_2,
\end{aligned}
\end{equation}
where
$$e_2:=\PV\int_0^1 \frac{\bar{B}_4[\frac{\alpha}{a}]-B_4}{4!}\frac{\ud ^4}{\ud  \alpha^4}\biggl(\frac{1}{\phi(\alpha)-\phi(\alpha_i)}\biggr)\ud\alpha.$$
Due to $\phi_\alpha(1)-\phi_\alpha(0)=0$ again, the third term in \eqref{617_4} becomes
\begin{equation}\label{K1}
\begin{aligned}
K_{1}:=&\frac{B_2}{2}\frac{\ud }{\ud\alpha}\biggl(\frac{1}{\phi(\alpha)-\phi(\alpha_i)}\biggr)\biggr|_{\alpha=1}^{\alpha=0}\\
=&\frac{B_2}{2}\phi_\alpha(0)\biggl(\frac{1}{(-\phi(\alpha_i))^2}-\frac{1}{(L-\phi(\alpha_i))^2}\biggr).
\end{aligned}
\end{equation}
Without loss of generality, we can also assume $\alpha_i\leq \frac{1}{2}$. Denote $p(\alpha):=\frac{\bar{B}_4[\frac{\alpha}{a}]-B_4}{4!},$ we have
\begin{equation}\label{617_5}
\begin{aligned}
e_2&=\PV\int_0^1 p(\alpha)\frac{\ud ^4}{\ud  \alpha^4}\biggl(\frac{G(\alpha)-G(\alpha_i)}{\alpha-\alpha_i}+\frac{G(\alpha_i)}{\alpha-\alpha_i}\biggr)\ud\alpha\\
&\leq C( \beta,\|h(0)\|_{W^{7,2}(I)})+\PV\int_0^1 c p(\alpha) \frac{\ud ^4}{\ud  \alpha^4}\biggl(\frac{1}{\alpha-\alpha_i}\biggr)\ud\alpha,
\end{aligned}
\end{equation}
where we used the differentiability of $G(\alpha)$.
For the last term in \eqref{617_5}, since $\alpha_i$ is the singular point, we do variable substitution to obtain
\begin{equation*}
\begin{aligned}
&\PV\int_0^1 c p(\alpha) \frac{\ud ^4}{\ud  \alpha^4}\biggl(\frac{1}{\alpha-\alpha_i}\biggr)\ud\alpha\\
=&\PV\int_{-\alpha_i}^{1-\alpha_i} c p(\alpha+\alpha_i) \frac{\ud ^4}{\ud  \alpha^4}\biggl(\frac{1}{\alpha}\biggr)\ud\alpha\\
=&\PV\int_{-\alpha_i}^{\alpha_i} c p(\alpha+\alpha_i) \frac{1}{ \alpha^5} \ud\alpha+\int_{\alpha_i}^{1-\alpha_i} c p(\alpha+\alpha_i) \frac{1}{ \alpha^5}\ud\alpha\\
= & \int_{\alpha_i}^{1-\alpha_i} c p(\alpha+\alpha_i) \frac{1}{ \alpha^5}\ud\alpha.
\end{aligned}
\end{equation*}
Here we used
$$\bar{B}_4\Big[\frac{\alpha+\alpha_i}{a}\Big]=\bar{B}_4\Big[\frac{\alpha}{a}\Big],$$
due to $\frac{\alpha_i}{a}$ is integer. Since $\bar{B}_4(x)$ is even, $c p(\alpha+\alpha_i) \frac{1}{ \alpha^5}$ is odd, so
the Cauchy principal value integral $\PV\int_{-\alpha_i}^{\alpha_i} c p(\alpha+\alpha_i) \frac{1}{ \alpha^5} \ud\alpha$ is zero.

Hence we get
\begin{equation}\label{e2}
|e_2|\leq C( \beta,\|h(0)\|_{W^{7,2}(I)}).
\end{equation}

On the other hand, \eqref{K2} and \eqref{K1} show that
\begin{equation*}
\begin{aligned}
K_1+K_{2}=&\sum_{k\in\mathbb{Z}}\frac{B_2}{2}\phi_\alpha(0)\biggl(\frac{1}{(kL-\phi(\alpha_i))^2}-\frac{1}{(L+kL-\phi(\alpha_i))^2}\biggr)=0.
\end{aligned}
\end{equation*}

Denote $e:=\sum_{k\in\mathbb{Z} \atop k\neq 0} e_1(k)+e_2$. Combining the calculations for $T_1$ and $T_2$, we obtain
\begin{equation*}
\begin{aligned}
\PV\int_{0}^{1}\frac{\pi}{L}\cot(\frac{\pi}{L}(\phi(\alpha)-\phi(\alpha_i)))\ud\alpha=\sum_{j\neq i,j=1}^{N}a \frac{\pi}{L}\cot(\frac{\pi}{L}(\phi_j-\phi(\alpha_i)))-\frac{a}{2}\frac{\phi_{\alpha\alpha}}{\phi_\alpha^2}\biggr|_{\alpha_i}+e a^4,
\end{aligned}
\end{equation*}
with $\lvert e\rvert\leq C( \beta,\|h(0)\|_{W^{7,2}(I)}).$
This concludes \eqref{lq4.2} and $|R_{1,i}|\leq a^4C( \beta,\|h(0)\|_{W^{7,2}(I)}).$
\end{proof}

Notice that change of variable from $h$ to $\phi$ does not effect the Cauchy principal value integral and that $h_x<0$. Then similar to \eqref{623_1}, by \eqref{09temp1} and variable substitution, we have
\begin{equation*}
\begin{aligned}
&\PV \int_0^1 \frac{\pi}{L}\cot \Big( \frac{\pi}{L} \big(\phi(\alpha_i)-\phi(\alpha)\big) \Big)\ud \alpha
=-\PV \int_0^1 \sum_{k\in \mathbb{Z}}\frac{1}{\phi(\alpha)-\phi(\alpha_i)-kL}\ud \alpha\\
=&-\PV\int_{-\infty}^{+\infty}\frac{1}{\phi(\alpha)-\phi(\alpha_i)}\ud \alpha
=\PV\int_{-\infty}^{+\infty}\frac{h_x}{x-\phi_i}\ud x\\
=&\PV\sum_{k\in\mathbb{Z}}\int_{-\frac{L}{2}+kL}^{\frac{L}{2}+kL}\frac{h_x}{x-\phi_i}\ud x
=\frac{\pi}{L} \PV \int_{-\frac{L}{2}}^{\frac{L}{2}}h_x \cot(\frac{x-\phi_i}{L}\pi)\ud x\\
=&-\pi H(h_x)|_{\phi_i}.
\end{aligned}
\end{equation*}

This, combined with Lemma \ref{l4.2}, leads to
\begin{lem}\label{ll}
Let $I_{1,i}$ be defined in \eqref{I123} and $v_1$ be function of $\alpha$ defined as
\begin{equation}\label{u1}
v_1(\alpha;\phi):=-\frac{\phi_{\alpha\alpha}}{L\phi_\alpha^2}.
\end{equation}
Then we have
\begin{equation}\label{I1}
I_{1,i}=-\lpi H(h_x)\biggr|_{\phi_i}+v_1(\alpha_i;\phi)a+R_{1,i},
\end{equation}
with $|R_{1,i}|\leq  a^4C( \beta,\|h(0)\|_{W^{7,2}(I)})$.
\end{lem}

We now turn to estimate $I_{3,i}$.
\begin{lem}\label{l4.3}
Let $I_{3,i}$ be defined in \eqref{I123} and $v_3$ be function of $\alpha$ defined as
\begin{equation}\label{u3}
v_3(\alpha;\phi):=\frac{-\frac{5}{2}\phi_{\alpha\alpha}^3-\frac{1}{4}\phi_\alpha^2\phi^{(4)}+2\phi_\alpha\phi_{\alpha\alpha}\phi^{(3)}}{\phi_\alpha^6}.
\end{equation}
Then we have
\begin{equation}\label{I3}
I_{3,i}=3{h_{xx}}{h_x}|_{\phi_i}+v_3(\alpha_i;\phi) a^2+R_{3,i},
\end{equation}
where $|R_{3,i}|\leq  a^4C( \beta,\|h(0)\|_{W^{7,2}(I)})$.
\end{lem}
\begin{proof} Using \eqref{fan} and Taylor expansion, it is similar to the proof of Lemma \ref{l4.1} that

\begin{equation*}
\begin{aligned}
I_{3,i}&=a^2 \biggl( \frac{1}{(\phi_{i+1}-\phi_i)^3}-\frac{1}{(\phi_i-\phi_{i-1})^3} \biggr)\\
&=\frac{\frac{2\phi_i-\phi_{i+1}-\phi_{i-1}}{a^2}}{(\frac{\phi_{i+1}-\phi
_i}{a})^3(\frac{\phi
_i-\phi
_{i-1}}{a})^3}\cdot \biggl((\frac{\phi_i-\phi_{i-1}}{a})^2+(\frac{\phi_{i+1}-\phi_{i}}{a})^2+(\frac{\phi_i-\phi_{i-1}}{a})(\frac{\phi_{i+1}-\phi_{i}}{a})\biggr)\\
&=\frac{\biggl(-\phi_{\alpha\alpha,i}-\frac{1}{12}\phi^{(4)}_i a^2-\frac{1}{6!}(\phi^{(6)}(\xi^+)+\phi^{(6)}(\xi^-)) a^4\biggr)\biggl(3\phi_{\alpha,i}^{2}+B_{1,i} a^2+B_{2,i} a^4\biggr)}{\phi_{\alpha,i}^{6}+C_{1,i} a^2+C_{2,i} a^4}\\
&=\Big(-\frac{3\phi_{\alpha\alpha}}{\phi_\alpha^4}\Big)_i+\biggl[\frac{-\frac{5}{2}\phi_{\alpha\alpha}^3-\frac{1}{4}\phi_\alpha^2\phi^{(4)}+2\phi_\alpha\phi_{\alpha\alpha}\phi^{(3)}}{\phi_\alpha^6}\biggr]_i a^2+C_{3,i} a^4\\
&=(3h_{xx}h_x)_i+\biggl[\frac{-\frac{5}{2}\phi_{\alpha\alpha}^3-\frac{1}{4}\phi_\alpha^2\phi^{(4)}+2\phi_\alpha\phi_{\alpha\alpha}\phi^{(3)}}{\phi_\alpha^6}\biggr]_i a^2+C_{3,i} a^4,
\end{aligned}
\end{equation*}
where
$$B_{1,i}=(\phi_\alpha \phi^{(3)}+\frac{1}{4}\phi_{\alpha\alpha}^2)_i,\quad |B_{2,i}|\leq c,$$
$$C_{1,i}=(-\frac{3}{4}\phi_\alpha^4 \phi_{\alpha\alpha}^2+\phi_{\alpha}^5 \phi^{(3)})_i,\quad |C_{2,i}|\leq c,\quad |C_{3,i}|\leq c.$$
Denote
\begin{equation*}
v_3(\alpha;\phi):=\frac{-\frac{5}{2}\phi_{\alpha\alpha}^3-\frac{1}{4}\phi_\alpha^2\phi^{(4)}+2\phi_\alpha\phi_{\alpha\alpha}\phi^{(3)}}{\phi_\alpha^6}.
\end{equation*}
We conclude the proof of lemma \ref{l4.3}.
\end{proof}

Denote
\begin{equation}\label{A624}
A(x;h):=\biggl(-\lpi H(h_x)+3{h_{xx}}{h_x}+\frac{h_{xx}}{h_x}\biggr)(x),
\end{equation}
and
$$ R_{4,i}:=R_{1,i}+R_{2,i}+R_{3,i}.$$
The above three
lemmas yield
\begin{lem}\label{l4.4}
For $\bar{f}_i$ defined in \eqref{fbar}, $v_1$ defined in \eqref{u1}, $v_2$ defined in \eqref{u2}, and $v_3$ defined in \eqref{u3}, we have
\begin{equation}\label{f624}
\bar{f}_i=A(\phi_i;h)+v_1(\alpha_i;\phi)a+(v_2+v_3)(\alpha_i;\phi) a^2+R_{4,i},
\end{equation}
where $|R_{4,i}|\leq a^4C( \beta,\|h(0)\|_{W^{7,2}(I)})$.
\end{lem}

Now we are ready to proof the main result of this section, Theorem \ref{th4.4}.

\begin{proof}[Proof of Theorem \ref{th4.4}.]
Step 1. To calculate $F_i$ in \eqref{Fi}, by \eqref{f624} in Lemma \ref{l4.4}, we first need to calculate
\begin{equation}\label{618_1}
\begin{aligned}
&\frac{A_{i+1}-A_i}{\phi_{i+1}-\phi_i}-\frac{A_i-A_{i-1}}{\phi_{i}-\phi_{i-1}}\\
=& A_{xx,i} \frac{\phi_{i+1}-\phi_{i-1}}{2}+A_{xxx,i}\frac{(\phi_{i+1}-\phi_{i-1})(\phi_{i+1}+\phi_{i-1}-2\phi_i)}{3!}+r_{1,i} a^4\\
=&-\phi_{\alpha,i} A_{xx,i} a+r_2(\alpha_i;\phi) a^3+r_{3,i} a^4,
\end{aligned}
\end{equation}
where $|r_{1,i}|,|r_{3,i}|\leq C( \beta,\|h(0)\|_{W^{7,2}(I)})$ and
$$r_2(\alpha;\phi):=\biggl(-\frac{1}{3}\phi^{(3)}(A_{xx}\circ\phi)-2\phi_\alpha \phi_{\alpha\alpha}\biggr)(\alpha).$$

Second,
 for any smooth function $v(\alpha)$ respect to $\alpha$, notice that
$$v_{i+1}-v_i=v_{\alpha,i}(\alpha_{i+1}-\alpha_i)+\frac{1}{2}v_{\alpha\alpha,i}(\alpha_{i+1}-\alpha_i)^2+\frac{1}{3!}v^{(3)}_i(\xi^+)(\alpha_{i+1}-\alpha_i)^3,$$
$$v_{i-1}-v_i=v_{\alpha,i}(\alpha_{i-1}-\alpha_i)+\frac{1}{2}v_{\alpha\alpha,i}(\alpha_{i-1}-\alpha_i)^2+\frac{1}{3!}v^{(3)}_i(\xi^-)(\alpha_{i-1}-\alpha_i)^3.$$
Then for other terms in \eqref{f624}, we have
\begin{equation}\label{618_2}
\begin{aligned}
&\frac{v_{i+1}-v_i}{\phi_{i+1}-\phi_i}-\frac{v_i-v_{i-1}}{\phi_{i}-\phi_{i-1}}\\
=&\frac{v_{i+1}-v_i}{\alpha_{i+1}-\alpha_i} \frac{h_{i+1}-h_i}{\phi_{i+1}-\phi_i}-\frac{v_i-v_{i-1}}{\alpha_{i}-\alpha_{i-1}} \frac{h_{i}-h_{i-1}}{\phi_{i}-\phi_{i-1}}\\
=&\biggl[v_{\alpha,i}-\frac{1}{2}v_{\alpha\alpha,i}a+\frac{1}{3!}v^{(3)}(\xi^+)a^{2}\biggr]
\biggl[h_{x,i}+h_{xx,i}\frac{\phi_{i+1}-\phi_i}{2}+\frac{1}{3!}h_{xxx}(\eta^+)(\phi_{i+1}-\phi_i)^2\biggr]\\
&-\biggl[v_{\alpha,i}+\frac{1}{2}v_{\alpha\alpha,i}a+\frac{1}{3!}v^{(3)}(\xi^-)a^{2}\biggr]
\biggl[h_{x,i}-h_{xx,i}\frac{\phi_{i}-\phi_{i-1}}{2}+\frac{1}{3!}h_{xxx}(\eta^-)(\phi_{i}-\phi_{i-1})^2\biggr]\\
=& r_4(\alpha_i;\phi) a+r_{5,i} a^2,
\end{aligned}
\end{equation}
where $|r_{5,i}|\leq C( \beta,\|h(0)\|_{W^{7,2}(I)})$, $\eta^+\in[\phi_i,\phi_{i+1}],\,\eta^-\in[\phi_{i-1},\phi_i]$ and
$$r_4(\alpha;\phi):=\biggl(\frac{v_\alpha \phi_{\alpha\alpha}}{\phi_\alpha^2}-\frac{v_{\alpha\alpha}}{\phi_\alpha}\biggr)(\alpha).$$

Denote
\begin{equation}\label{ralpha}
r_0(\alpha;\phi):=\biggl(\frac{v_{1\alpha} \phi_{\alpha\alpha}}{\phi_\alpha^2}-\frac{v_{1\alpha\alpha}}{\phi_\alpha}\biggr)(\alpha),
\end{equation}
and
\begin{equation}
r(\alpha;\phi):=\biggl(\frac{v_{2\alpha} \phi_{\alpha\alpha}}{\phi_\alpha^2}-\frac{v_{2\alpha\alpha}}{\phi_\alpha}+\frac{v_{3\alpha} \phi_{\alpha\alpha}}{\phi_\alpha^2}-\frac{v_{3\alpha\alpha}}{\phi_\alpha}\biggr)(\alpha)+r_2(\alpha;\phi).
\end{equation}
Thus for $F_i$ in \eqref{Fi}, combining \eqref{618_1} and \eqref{618_2}, we get
\begin{equation}\label{618_3}
\begin{aligned}
F_i=&-\frac{A_{xx}}{h_x}(\phi_i)+r_0(\alpha_i;\phi)a+r(\alpha_i;\phi) a^2+ \frac{R_{4,i+1}-2R_{4,i}+R_{4,i-1}}{a^2}(h_x(\phi_i)+r_{6,i} a)\\
=&-\frac{A_{xx}}{h_x}(\phi_i)+r_0(\alpha_i;\phi)a+r(\alpha_i;\phi) a^2+ R_{5,i}a^2,
\end{aligned}
\end{equation}
where $|r_{6,i}|\leq C( \beta,\|h(0)\|_{W^{7,2}(I)}),$ $A(x;h)$ defined in \eqref{A624}. To obtain $|R_{5,i}|\leq  C( \beta,\|h(0)\|_{W^{7,2}(I)}),$ here we also used $|R_{4,i}|\leq  a^4 C( \beta,\|h(0)\|_{W^{7,2}(I)})$
due to Lemma \ref{l4.4}.

Denote
\begin{align}\label{0502t1}
R_i:&=r(\alpha_i;\phi) + R_{5,i}.
\end{align}
For $a$ small enough, we have $|R_i|\leq |r(\alpha_i;\phi)| +| R_{5,i}|\leq C( \beta,\|h(0)\|_{W^{7,2}(I)}),$ .
Finally, comparing \eqref{618_3} with \eqref{4.2}, we conclude \eqref{consistency1}.

Step 2. Now using \eqref{consistency1} and Lemma \ref{l4.4}, we can claim
\begin{equation}\label{617_6}
\sum_{i=1}^N \bar{f}_i\biggl(F_i-\frac{\ud  \phi_i}{\ud t}\biggr)\leq C( \beta,\|h(0)\|_{W^{7,2}(I)}),
\end{equation}
where $C( \beta,\|h(0)\|_{W^{7,2}(I)})$ depends on $\beta,\|h(0)\|_{W^{7,2}(I)}$.

From \eqref{617_6}, multiplying $\bar{f}_i$ in \eqref{consistency1} and summation by parts show that
$$\frac{\ud E^N(\phi)}{\ud t}+\sum_{i=1}^N\frac{\biggl(\bar{f}_{i+1}(\phi)-\bar{f}_i(\phi)\biggr)^2}{\phi_{i+1}-\phi_i}\leq C( \beta,\|h(0)\|_{W^{7,2}(I)})a,$$
Then by \eqref{y}, we have
$$\frac{\ud E^N(\phi)}{\ud t}+a \sum_{i=1}^N \biggl(\frac{\bar{f}_{i+1}(\phi)-\bar{f}_i(\phi)}{a}\biggr)^2\leq C( \beta,\|h(0)\|_{W^{7,2}(I)})a,$$
which completes the proof of Theorem \ref{th4.4}.
\end{proof}

\section{Convergence and the proof of Theorem \ref{thm_main}}\label{sec7}

In this section, our goal is to prove Theorem \ref{thm_main}. The main idea is to first construct an auxiliary solution with high-order consistency (see Section \ref{sec6.3}), and then prove the convergence rate for the auxiliary solution, which helps us obtain the convergence rate for the original PDE solution.

\subsection{Stability of linearized $x$-ODE}\label{sec6.2}

First of all, we devote to study the stability of linearized ODE, which is important when we estimate the convergence rate for the auxiliary solution.
The procedure here is analogous to the stability result of linearized $\phi$-PDE; see Section \ref{sec6.1}.

For vector $x,\,y$ satisfying \eqref{ODE}, set $x=y+\varepsilon z$. We also assume $y_i(t)=\phi(\alpha_i,t),$ and $\phi$ is the solution of \eqref{phi3} satisfying \eqref{04_temp1} and \eqref{04_temp2}.
Denote
\begin{equation}\label{619_1}
M_i=\frac{1}{y_{i+1}-y_i}+\frac{a^2}{(y_{i+1}-y_i)^3}-\frac{1}{y_i-y_{i-1}}-\frac{a^2}{(y_i-y_{i-1})^3}-\frac{2}{L}\sum_{j\neq i}\frac{a}{y_j-y_{i}},
\end{equation}
and
\begin{equation}\label{619_2}
T_i=-\frac{z_{i+1}-z_{i}}{(y_{i+1}-y_{i})^2}-3a^2\frac{z_{i+1}-z_{i}}{(y_{i+1}-y_{i})^4}+\frac{z_{i}-z_{i-1}}{(y_{i}-y_{i-1})^2}+3a^2\frac{z_{i}-z_{i-1}}{(y_{i}-y_{i-1})^4}+\frac{2}{L}\sum_{j\neq i}\frac{a(z_j-z_i)}{(y_j-y_i)^2}.
\end{equation}
Then $z$ satisfies the following linearized equation
\begin{align}\label{linear_2}
\frac{\ud }{\ud t}z_i&=\frac{1}{a}\biggl(\frac{T_{i+1}-T_i}{y_{i+1}-y_i}-\frac{T_i-T_{i-1}}{y_i-y_{i-1}}\biggr)-\frac{1}{a}\biggl[\frac{z_{i+1}-z_i}{(y_{i+1}-y_i)^2}(M_{i+1}-M_i)-\frac{z_i-z_{i-1}}{(y_i-y_{i-1})^2}(M_i-M_{i-1})\biggr].
\end{align}

\begin{prop}\label{prop19_1}
Assume $z(0)\in \ell^2$ and $m_1,m_2>0$ defined in \eqref{04_temp2}. Let $T_m>0$ be the maximal existence time for strong solution $\phi$ in \eqref{04_temp1}. The linearized equation \eqref{linear_2} is stable in the sense
\begin{equation}\label{linear_z}
\|z(t)\|_{\ell^2}\leq C(m_1,m_2,T_m)\|z(0)\|_{\ell^2},\text{ for }t\in[0,T_m],
\end{equation}
where $C(m_1,m_2,T_m)$ is a constant depending only on $m_1,\,m_2,$ and $T_m.$
\end{prop}

\begin{proof}
Step 1. Similar to the proof of Proposition \ref{prop17_1}, first we study the linearized system for \eqref{ODE} without the Hilbert transform term $-\frac{2}{L}\sum_{j\neq i}\frac{a}{x_j-x_i}$. Thus $M_i,\,T_i$ in \eqref{619_1} and \eqref{619_2} become
\begin{equation*}
M_i=\frac{1}{y_{i+1}-y_i}+\frac{a^2}{(y_{i+1}-y_i)^3}-\frac{1}{y_i-y_{i-1}}-\frac{a^2}{(y_i-y_{i-1})^3},
\end{equation*}
and
\begin{equation*}
T_i=-\frac{z_{i+1}-z_{i}}{(y_{i+1}-y_{i})^2}-3a^2\frac{z_{i+1}-z_{i}}{(y_{i+1}-y_{i})^4}+\frac{z_{i}-z_{i-1}}{(y_{i}-y_{i-1})^2}+3a^2\frac{z_{i}-z_{i-1}}{(y_{i}-y_{i-1})^4}.
\end{equation*}

Since $z_{i+N}=z_i$, multiplying both sides of \eqref{linear_2} by $az_i$ and taking summation by parts, we have
\begin{align*}
\sum_{i=1}^{N}az_i\dot{z}_i&=-\sum_{i=1}^N \frac{z_{i+1}-z_i}{y_{i+1}-y_i}(T_{i+1}-T_i)+\sum_{i=1}^N(z_{i+1}-z_i)\frac{z_{i+1}-z_i}{(y_{i+1}-y_i)^2}(M_{i+1}-M_i)\\
&=-a\sum_{i=1}^N \frac{z_{i+1}-z_i}{a}\frac{\frac{T_{i+1}}{\frac{y_{i+1}-y_i}{a}}-\frac{T_{i}}{\frac{y_{i}-y_{i-1}}{a}}}{a}-a\sum_{i=1}^N \frac{z_{i+1}-z_i}{a}\frac{\frac{T_{i}}{\frac{y_{i}-y_{i-1}}{a}}-\frac{T_{i}}{\frac{y_{i+1}-y_{i}}{a}}}{a}\\
&+a\sum_{i=1}^N \biggl(\frac{z_{i+1}-z_i}{a}\biggr)^2 \frac{1}{(\frac{y_{i+1}-y_i}{a})^2}\frac{M_{i+1}-M_i}{a}\\
&=I_1+I_2+I_3.
\end{align*}

Next, we will estimate $I_1,\,I_2,\,I_3$ one by one. First, we deal with
\begin{align*}
I_1&=-a\sum_{i=1}^N \frac{z_{i+1}-z_i}{a}\frac{\frac{T_{i+1}}{\frac{y_{i+1}-y_i}{a}}-\frac{T_{i}}{\frac{y_{i}-y_{i-1}}{a}}}{a}\\
&=a\sum_{i=1}^N\frac{T_i}{\frac{y_i-y_{i-1}}{a}}\frac{\frac{z_{i+1}-z_i}{a}-\frac{z_{i}-z_{i-1}}{a}}{a}.
\end{align*}
We can see
\begin{align*}
T_i=&a^2\frac{z_{i+1}-2z_i+z_{i-1}}{a^2}\biggl(-\frac{1}{(y_{i+1}-y_i)^2}-\frac{3a^2}{(y_{i+1}-y_i)^4}\biggr)\\
&+a\biggl[-\frac{1}{(y_{i+1}-y_i)^2}-\frac{3a^2}{(y_{i+1}-y_i)^4}+\frac{1}{(y_i-y_{i-1})^2}+\frac{3a^2}{(y_i-y_{i-1})^4}\biggr]\frac{z_i-z_{i-1}}{a}.
\end{align*}
Due to Young's inequality, for any $\varepsilon>0$, we have
\begin{align}\label{young}
a\sum_{i=1}^N\biggl(\frac{z_{i+1}-z_{i}}{a}\biggr)^2&=-a\sum_{i=1}^N z_{i}\frac{z_{i+1}-2z_{i}+z_{i-1}}{a^2}\nonumber\\
&\leq a\sum_{i=1}^N\biggl(\frac{1}{4\varepsilon}z_i^2+\varepsilon\biggl(\frac{z_{i+1}-2z_i+z_{i-1}}{a^2}\biggr)^2\biggr) .
\end{align}
Besides, due to $y_i(t)=\phi(\alpha_i,t)$, we have
\begin{align*}
a\biggl[-\frac{1}{(y_{i+1}-y_i)^2}-\frac{3a^2}{(y_{i+1}-y_i)^4}+\frac{1}{(y_i-y_{i-1})^2}+\frac{3a^2}{(y_i-y_{i-1})^4}\biggr] \frac{a}{y_i-y_{i-1}}\leq C_0(m_1,m_2),
\end{align*}
$$\biggl(-\frac{1}{(y_{i+1}-y_i)^2}-\frac{3a^2}{(y_{i+1}-y_i)^4}\biggr) a^2 \frac{a}{y_i-y_{i-1}}\leq -C(m_2)$$
for $a$ small enough.

Then for $I_1$, we have
\begin{align*}
I_1&=a\sum_{i=1}^N\frac{T_i}{\frac{y_i-y_{i-1}}{a}}\frac{\frac{z_{i+1}-z_i}{a}-\frac{z_{i}-z_{i-1}}{a}}{a}\\
&\leq C_1(m_1,m_2)a\sum_i z_i^2-\frac{3}{4}C(m_2)a\sum_i\biggl(\frac{z_{i+1}-2z_i+z_{i-1}}{a^2}\biggr)^2.
\end{align*}

Let us keep in mind that functions, such as $M_i,$ involving only $\frac{y_{i+1}-y_i}{a}$ can be bounded by a constant depending only on $m_1,\,m_2$.
Then similar to the estimate for $I_1$, together with \eqref{young}, we have
$$I_2\leq  C_{2}(m_1,m_2)a\sum_i z_i^2
+\frac{1}{4}C(m_2)a\sum_i\biggl(\frac{z_{i+1}-2z_i+z_{i-1}}{a^2}\biggr)^2,$$
and
$$I_3\leq  C_3(m_1,m_2)a\sum_i z_i^2
+\frac{1}{4}C(m_2)a\sum_i\biggl(\frac{z_{i+1}-2z_i+z_{i-1}}{a^2}\biggr)^2.$$
Here $C_i(m_1,m_2),\,i=0,1,2,3$ are positive constants depending only on $m_1,\,m_2.$

Combining estimates for $I_1,\, I_2,\,I_3$, we have
$$\frac{\ud \|z(t)\|_{\ell^2}^2}{\ud t}+\frac{1}{4}C(m_2)a\sum_i\biggl(\frac{z_{i+1}-2z_i+z_{i-1}}{a^2}\biggr)^2\leq C(m_1,m_2)\|z\|^2_{\ell^2}.$$
Then Gr\"{o}nwall's inequality yields \eqref{linear_z}.

Step 2. Now we consider Hilbert transform term $-\frac{2}{L}\sum_{j\neq i}\frac{a}{x_j-x_i}$. Then the terms $M_i,\,T_i$ in \eqref{linear_2} become \eqref{619_1} and \eqref{619_2}.

 First Lemma \ref{l4.2} and Lemma \ref{ll} show that $\sum_{j\neq i}\frac{a}{y_j-y_i}$ can be estimated by $C(m_1,m_2)$ and $\PV\int_0^1 \cot\frac{\pi}{L}(\phi(\alpha)-\phi(\beta))\ud\beta$.

 Second, from the proof of Lemma \ref{l4.2}, we know $a\sum_{j\neq i}\frac{z_j-z_i}{(y_j-y_i)^2}$ can be estimated by $C(m_1,m_2)$ and $\PV\int_0^1 \sec^2\frac{(\phi(\alpha)-\phi(\beta))\pi}{L}(\psi(\alpha)-\psi(\beta))\ud\beta$,
 where $\psi$ is the piecewise-cubic interpolant of $z$.

 Then using the same arguments in step 2 of the proof of Proposition \ref{prop17_1}, we can conclude \eqref{linear_z}.
\end{proof}

\subsection{Construction of solution with high-order truncation error}\label{sec6.3}

From now on, we proceed under the same hypothesis of Theorem \ref{thm_main}, i.e. we assume for some $\beta<0$, the initial datum $h(0)$ smooth enough and satisfies
\begin{equation}
 h_x(0)\leq \beta <0.
\end{equation}

By Theorem \ref{local_h} and Proposition \ref{h-phi}, for some constant $m\in\mathbb{N}$ large enough, we know there exists $T_m>0$, such that
\begin{equation}\label{04_temp3}
\phi(\alpha,t)\in C([0,T_m];C^{m}[0,1])
\end{equation}
is the strong solution to \eqref{phi3}.
Obviously, there exist $M>0$, whose values depend only on $\beta$ and $\|h(0)\|_{W^{m,2}}$, such that
\begin{equation}\label{04_temp4}
\begin{aligned}
\phi_\alpha\leq \frac{\beta}{2}<0,\quad \lvert\phi^{(i)}\rvert \leq M, \text{ for }1\leq i\leq m.
\end{aligned}
\end{equation}

Recalling equation \eqref{phi3}, we define $F(\phi):C^{\infty}[0,1]\rightarrow C^\infty[0,1]$ as an operator
$$F(\phi):=-\partial_\alpha(\frac{1}{\phi_\alpha}(\frac{\delta E}{\delta \phi})_\alpha).$$
Then we have
\begin{equation}\label{619phi}
\phi_t=F(\phi).
\end{equation}
For $F_i$ defined in \eqref{Fi}, denote
$$F_N:=\{F_i,\,i=1,\cdots,N\},\quad  r_N(\phi):=\{r_0(\alpha_i;\phi),\,i=1,\cdots,N\},$$
where $r_0(\alpha;\phi)$ is the function defined in \eqref{ralpha0}.
Then for $\phi_N=\{\phi_i,\,i=1,\cdots,N\}$, Theorem \ref{th4.4} shows that
$$\dot{\phi}_N=F_N(\phi_N)+r_N(\phi)a+O(a^2).$$

Now we want to construct $y=\phi+a \psi$, for $\psi$ satisfying the same regularity with $\phi$, such that $y$ has a higher truncation error than $\phi$.
In fact, we state
\begin{prop}\label{prop19_2}
Let $T_m>0$ in \eqref{04_temp3} and $\phi$ be the solution of \eqref{619phi}. Then there exists $\psi$ smooth enough such that
$\|\psi(\cdot,t)\|_{L^2([0,1])}\text{is uniformly bounded for }t\in[0,T_m],$
and
\begin{equation}\label{620_1}
y(\alpha,t)=\phi(\alpha,t)+a\psi(\alpha,t)
\end{equation}
satisfies the ODE system \eqref{ODE} till order $O(a^7)$, i.e. the nodal values $y_N=\{y(\alpha_i,t),\,i=1,\cdots,N\}$ satisfy
\begin{equation}\label{619_3}
\dot{y}_N=F_N(y_N)+O(a^7).
\end{equation}
\end{prop}

\begin{proof}

To simplify the calculation, first we show there exists $\psi$ such that
\begin{equation}\label{7.8}
\dot{y}_N=F_N(y_N)+O(a^2).
\end{equation}

For $y_N=\phi_N+a\psi_N,$ where $\psi_N$ is the nodal values of $\psi,$
Theorem \ref{th4.4} shows that
$$F_N(y_N)=F_N(\phi_N+a\psi_N)=F(\phi+a \psi)|_{\alpha=\alpha_i}-r_N(\phi+a \psi)a-O(a^2).$$
Hence $y_N$ satisfies
$$\dot{y}_N-F_N(y_N)=a \dot{\psi}_N+[F(\phi)-F(\phi+a \psi)]|_{\alpha=\alpha_i}+r_N(\phi+a \psi)a+O(a^2).$$
Now by Proposition \ref{prop17_1}, we can choose $\psi$ to be the solution of \eqref{619phi}'s linearized system
\begin{equation}
\psi_t=-\partial_\alpha \biggl( -\frac{\psi_\alpha}{\phi_\alpha^2} \partial_\alpha A+\frac{\partial_\alpha B}{\phi_\alpha} \biggr)-r_0(\phi),
\end{equation}
where $A,\,B$ are defined in \eqref{A_H} and \eqref{B_H}. After that, \eqref{7.8} holds.

To obtain higher order truncation error construction, we can repeat above processes to get higher order corrections. We omit the details here.
\end{proof}

\subsection{Convergence of ODE and PDE system}\label{sec6.4}

In this section, we will combine above results and complete the proof of Theorem \ref{thm_main}.

\begin{proof}[Proof of Theorem \ref{thm_main}.]
Assume $\phi$ is the strong solution of \eqref{phi3} satisfying \eqref{04_temp3} and \eqref{04_temp4} with maximal existence time $T_m>0$. Let $\beta,\,M$ be constants in equation \eqref{04_temp4}.
  Recall vector $x(t)=\{x_i(t); \,i=1,\cdots,N\}$ is the solution of \eqref{ODE}, and with slight abuse of notation, denote $y(t):=\{y(\alpha_i,t); \,i=1,\cdots,N\}$ being the constructed vector value function $y_N$ in Proposition \ref{prop19_2}. We will first obtain the convergence rate for $x,\,y$ in Step 1, 2, and then obtain the convergence rate for $x,\,\phi$ in Step 3.

Step 1. We first claim that under the \textit{a-priori} assumption
\begin{equation}\label{pri_xy}
\|x(t)-y(t)\|_{\ell^\infty}\leq a^{6+\frac{1}{3}}, \text{ for }t\in[0,T_m],
\end{equation}
we have
\begin{equation}\label{l2_xy}
\|x(t)-y(t)\|_{\ell^2}\leq C(\beta,M,T_m)a^7,\text{ for }t\in[0,T_m],
\end{equation}
where $C(\beta,M,T_m)$ is a constant depending only on $\beta,\,M,\,T_m.$ We will verify the \textit{a-priori} assumption \eqref{pri_xy} in Step 2.

In fact, from Proposition \ref{prop19_2}, we know $y$ has $a^7$-order consistence error, i.e.
$$\frac{\ud (y-x)}{\ud t}=F_N(y)-F_N(x)+O(a^7).$$
Denote the inner product for $x,\,y$ as
$$\langle x,y \rangle :=\sum_{i=1}^N a x_i y_i.$$
Then for $\beta,\,M$ defined in \eqref{04_temp4}, we have
\begin{equation}\label{03_7.7}
\begin{aligned}
\langle x-y,\dot{x}-\dot{y} \rangle =&\langle x-y, \nabla F_N(y)(x-y) \rangle +\langle x-y,(x-y)\nabla^2 F_N(y)(x-y)^T \rangle \\
&+C(\beta,M)\langle x-y, a^7\rangle,
\end{aligned}
\end{equation}
where $C(\beta,M)$ depends only on $\beta,\,M$.

For the second term in \eqref{03_7.7}, we can see
\begin{equation}\label{04_7.8}
\begin{aligned}
&\langle x-y,(x-y)\nabla^2 F_N(y)(x-y)^T \rangle\\
\leq & \|x-y\|_{\ell^2} \|\sum_{i,j=1}^N(x_i-y_i)(x_j-y_j)\partial_{ij}F_N\|_{\ell^2}\\
\leq & \|x-y\|_{\ell^2}^2 \|x-y\|_{\ell^\infty}\sqrt{\sum_{k=1}^N\Bigl(\sum_{i=1}^N \bigl(\sum_{j=1}^N (\partial_{ij}F_k)^2 \bigr)^{\frac{1}{2}} \Bigr)^2}\\
\leq & \|x-y\|_{\ell^2}^2 \|x-y\|_{\ell^\infty} N \max_{k}\Biggl(\sqrt{\sum_{i=1}^N\sum_{j=1}^N(\partial_{ij}F_k)^2}\Biggr),\\
\end{aligned}
\end{equation}
where we used H\"older's inequality in the last step.

Now keep in mind that functions involving only $\frac{y_{i+1}-y_i}{a}$  can be bounded by a constant depending only on $\beta,\,M$, and that
$$\Bigl\lvert \frac{1}{y_j-y_i}\Bigr\rvert\leq \max\{\frac{1}{y_{i+1}-y_i},\frac{1}{y_i-y_{i-1}}\}.$$
We can start to estimate the term $\max_{k}\Biggl(\sqrt{\sum_i\sum_j(\partial_{ij}F_k)^2}\Biggr).$

For $k=1,\cdots,N$, denote
\begin{align*}
Q_k:= &\frac{1}{y_{k+1}-y_k}\Biggl[-\sum_{\ell\neq k+1}\frac{a}{y_{\ell}-y_{k+1}}+\sum_{\ell\neq k}\frac{a}{y_{\ell}-y_{k}}+\biggl(\frac{1}{y_{k+2}-y_{k+1}}-2\frac{1}{y_{k+1}-y_{k}}+\frac{1}{y_{k}-y_{k-1}}\biggr)\\
&+\biggl(\frac{a^2}{(y_{k+2}-y_{k+1})^3}-2\frac{a^2}{(y_{k+1}-y_{k})^3}+\frac{a^2}{(y_{k}-y_{k-1})^3}\biggr)\Biggr].
\end{align*}
Then $F_k=\frac{1}{a}(Q_k-Q_{k-1}),$ and
$$(\partial_{ij}F_k)^2\leq \frac{1}{a} [(\partial_{ij}Q_k)^2+(\partial_{ij}Q_{k-1})^2]. $$

First calculate $\partial_{i}Q_k,$ for $k=1,\cdots,N$.
\begin{align*}
\partial_i Q_k=\left\{
                 \begin{array}{ll}
                   \dpstyle{\frac{a}{y_{k+1}-y_k}\Big[\frac{1}{(y_i-y_{k+1})^2}-\frac{1}{(y_i-y_k)^2} \Big]}, & \text{ for }\left.
                                                                                                                              \begin{array}{ll}
                                                                                                                                1\leq i\leq k-2, \\
                                                                                                                                k+3\leq i\leq N;
                                                                                                                              \end{array}
                                                                                                                            \right.\\
                   ~\\
                    \dpstyle{\frac{a}{y_{k+1}-y_k}\Big[\frac{1}{(y_{k-1}-y_{k+1})^2}-\frac{1}{(y_{k-1}-y_k)^2} \Big]} \\  \dpstyle{+\frac{1}{y_{k+1}-y_k}\frac{1}{(y_k-y_{k-1})^2}+\frac{1}{y_{k+1}-y_k}\frac{3a^2}{(y_k-y_{k-1})^4},}
                        &  \text{ for }i=k-1;\\
~\\
                    \dpstyle{\frac{a}{(y_{k+1}-y_k)^3}-\frac{4}{(y_{k+1}-y_k)^3}-\frac{8a^2}{(y_{k+1}-y_k)^5}}\\
                    \dpstyle{-\frac{y_{k+1}-2y_k+y_{k-1}}{(y_{k+1}-y_k)^2(y_k-y_{k-1})^2}-a^2\frac{3(y_{k+1}-y_k)-(y_k-y_{k-1})}{(y_{k+1}-y_k)^2(y_k-y_{k-1})^4}},
                    &  \text{ for }i=k;\\
~\\
                    \dpstyle{-\frac{a}{(y_{k+1}-y_k)^3}+\frac{4}{(y_{k+1}-y_k)^3}+\frac{8a^2}{(y_{k+1}-y_k)^5}}\\
                    \dpstyle{+\frac{y_{k+2}-2y_{k+1}+y_{k}}{(y_{k+1}-y_k)^2(y_{k+2}-y_{k+1})^2}+a^2\frac{(y_{k+2}-y_{k+1})-3(y_{k+1}-y_{k})}{(y_{k+1}-y_k)^2(y_{k+2}-y_{k+1})^4}},
                    &  \text{ for }i=k+1;\\
~\\
                \dpstyle{\frac{a}{y_{k+1}-y_k}\Big[\frac{1}{(y_{k+2}-y_{k+1})^2}-\frac{1}{(y_{k+2}-y_k)^2} \Big]} \\  \dpstyle{-\frac{1}{y_{k+1}-y_k}\frac{1}{(y_{k+2}-y_{k+1})^2}-\frac{1}{y_{k+1}-y_k}\frac{3a^2}{(y_{k+2}-y_{k+1})^4}},
                        &  \text{ for }i=k+2.\\
                 \end{array}
               \right.
\end{align*}
Hence
\begin{align*}
\partial_{ij}Q_k=\begin{array}{cc}
                \, & \begin{array}{cccccccccc}
                       j=1 & \cdots &\,k-2 & \,k-1 & \,\,j=k & \,k+1 & \,k+2 & \,\,\cdots & \,N-1 & j=N
                     \end{array}
\\
                \begin{array}{c}
                  i=1 \\
                  \vdots \\
                  k-2 \\
                  k-1 \\
                  k \\
                  k+1 \\
                  k+2 \\
                  k+3 \\
                  \vdots \\
                  i=N
                \end{array}
 & \left(
     \begin{array}{cccccccccc}
       O(\frac{1}{a^3}) & \textbf{0} & 0 & 0 & O(\frac{1}{a^3}) & O(\frac{1}{a^3}) & 0 & \textbf{0} & 0 & 0 \\
       \textbf{0} & \ddots & \textbf{0} & \textbf{0} & O(\frac{1}{a^3}) & O(\frac{1}{a^3}) & \textbf{0} & \textbf{0} & \textbf{0} & \textbf{0}\\
       0 & \textbf{0} & O(\frac{1}{a^3}) & 0 & O(\frac{1}{a^3}) & O(\frac{1}{a^3}) & 0 & \textbf{0} & 0 & 0 \\
       0 & \textbf{0} & 0 & O(\frac{1}{a^4}) & O(\frac{1}{a^4}) & O(\frac{1}{a^4}) & 0 & \textbf{0} & 0 & 0 \\
       0 & \textbf{0} & 0 & O(\frac{1}{a^4}) & O(\frac{1}{a^4}) & O(\frac{1}{a^4}) & 0 & \textbf{0} & 0 & 0  \\
       0 & \textbf{0} & 0 & 0 & O(\frac{1}{a^4}) & O(\frac{1}{a^4}) & O(\frac{1}{a^4}) & \textbf{0} & 0 & 0 \\
       0 & \textbf{0} & 0 & 0 & O(\frac{1}{a^4}) & O(\frac{1}{a^4}) & O(\frac{1}{a^4}) & \textbf{0} & 0 & 0 \\
       0 & \textbf{0} & 0 & 0 & O(\frac{1}{a^3}) & O(\frac{1}{a^3}) & 0 & O(\frac{1}{a^3}) & 0 & 0 \\
       \textbf{0} & \textbf{0} & \textbf{0} & \textbf{0} & O(\frac{1}{a^3}) & O(\frac{1}{a^3}) & \textbf{0} & \textbf{0} & \ddots & \textbf{0} \\
       0 & \textbf{0} & 0 & 0 & O(\frac{1}{a^3}) & O(\frac{1}{a^3}) & 0 & \textbf{0} & 0 & O(\frac{1}{a^3})
     \end{array}
   \right)
              \end{array},
\end{align*}
 where
$\{\partial_{ij}Q_k\}_{i,j=1,\cdots,k-2}$ and
$\{\partial_{ij}Q_k\}_{i,j=k+3,\cdots,N}$ are diagonal matrixes with
$O(\frac{1}{a^3})$ main diagonal entries and the bold zeros \textbf{0} represent zero matrices with corresponding dimensions.

For $Q_{k-1}$, we have a similar Hessian matrix.
Notice that only three terms in one row are nonzero and that only at most four terms in one column are order $\frac{1}{a^4}.$
Hence for $a$ small enough, we have
$$\max_{k}\Biggl(\sqrt{\sum_i\sum_j(\partial_{ij}F_k)^2}\Biggr)\leq C(\beta,M)\frac{1}{a^5}.$$
where $C(\beta,M)$ is a constant depending only on $\beta,\,M.$

Then from \eqref{04_7.8} and the \textit{a-priori} condition \eqref{pri_xy}, we have
$$\langle x-y,(x-y)\nabla^2 F_N(y)(x-y)^T \rangle\leq C(\beta,M)a^{\frac{1}{3}}\|x-y\|_{\ell^2}^2. $$
Combining this with \eqref{03_7.7}, together with linearized stability in Proposition \ref{prop19_1},
gives
\begin{align*}
\frac{\ud \|x-y\|_{\ell^2}^2}{\ud t}&\leq C(\beta,M)\|x-y\|^2_{\ell^2}+C(\beta,M)a^7\|x-y\|_{\ell^2}.
\end{align*}
Therefore by Gr\"{o}nwall's inequality,
we obtain
\begin{equation}\label{temp619}
\|x(t)-y(t)\|_{\ell^2}\leq C(\beta,M,T_m)(\|x(0)-y(0)\|_{\ell^2}+ a^7), \text{ for }t\in[0,T_m],
\end{equation}
where $C(\beta,M,T_m)$ is a constant depending only on $\beta,\,M,\,T_m.$ We choose initial data of $y$ such that $y(0)=x(0)$, so \eqref{temp619}
leads to \eqref{l2_xy}.

Step 2.  Now we need to verify the \textit{a-priori} assumption \eqref{pri_xy} is true for $t\in[0,T_m].$ In fact,
$$\|x(t)-y(t)\|_{\ell^\infty}\leq \frac{\|x(t)-y(t)\|_{\ell^2}}{\sqrt{a}}\leq C(\beta,M,T_m)a^{7-\frac{1}{2}}<<a^{6+\frac{1}{3}},$$
for $a$ small enough, $t\in[0,T_m]$.
Hence \eqref{l2_xy} actually verifies the \textit{a-priori} condition \eqref{pri_xy}.

Step 3.
For the exact strong solution $\phi$ of \eqref{phi3}, recall the nodal values $\phi_N=\{\phi_i,\,i=1,\cdots,N\}.$
By Proposition \ref{prop19_2}, we know that the constructed function $y$ in \eqref{620_1} satisfies
$$\|y(t)-\phi_N(t)\|_{\ell^2}=\|a\psi_N(t)\|_{\ell^2}\leq c a, \text{ for }t\in[0,T_m],$$
where we used $\psi(t)$, defined in Proposition \ref{prop19_2}, is uniformly bounded.
This, together with \eqref{l2_xy}, shows that
\begin{equation}
\|x(t)-\phi_N(t)\|_{\ell^2}\leq \|x(t)-y(t)\|_{\ell^2}+\|y(t)-\phi_N(t)\|_{\ell^2}\leq C(\beta,M,T_m) a, \text{ for }t\in[0,T_m],
\end{equation}
where $C(\beta,M,T_m)$ is a constant depending only on $\beta,\,M,\,T_m.$
This completes the proof of the Theorem \ref{thm_main}.
\end{proof}

\section*{Acknowledgements}
We would like to thank the support by the National Science
  Foundation under grants DMS-1514826 (JGL), DMS-1454939 (JL), and
  also through the research network KI-Net RNMS11-07444. We thank Dionisios Margetis,
  Jeremy Marzuola, Yang Xiang, Aaron N.K. Yip for helpful
  discussions.

\newpage


\end{document}